\documentclass[11pt]{article}
\usepackage[numbers,sort&compress]{natbib}
\usepackage{enumerate}
\usepackage{tikz}
\usepackage{pgfplots}
\usepackage{amscd}
\usepackage{amsmath}
\usepackage{latexsym}
\usepackage{amsfonts}
\usepackage{amssymb}
\usepackage{amsthm}
\usepackage{verbatim}
\usepackage{mathrsfs}
\usepackage{enumerate}
\usepackage{hyperref}

\theoremstyle{plain}
\theoremstyle{definition}\newtheorem{theorem}{Theorem}[section]
\theoremstyle{plain}\newtheorem{lemma}[theorem]{Lemma}
\theoremstyle{plain}\newtheorem{coro}[theorem]{Corollary}
\theoremstyle{plain}
\theoremstyle{remark}\newtheorem{remark}{Remark}[section]
\usepackage{xcolor}

\newcommand{\norm}[1]{\left\|#1\right\|}
\newcommand{\Div}{\mathrm{div}\,}
\newcommand{\B}{\Big}

\newcommand{\be}{\begin{equation}}
\newcommand{\ee}{\end{equation}}
 \newcommand{\ba}{\begin{aligned}}
 \newcommand{\ea}{\end{aligned}}

\newcommand{\fbxo}{\int_{B_{\varepsilon}(x)}\!\!\!\!\!\!\!\!\!\!\!\!\!\!\!\! -~\,\,~\,~\,}

\newcommand{\fbxoo}{\int_{B_{\varepsilon}(0)}\!\!\!\!\!\!\!\!\!\!\!\!\!\!\!\! -~\,~\,~\,~\,}

  \newcommand{\f}{\frac}
    
  \newcommand{\ben}{\begin{enumerate}}
   \newcommand{\een}{\end{enumerate}}

\newcommand{\ti}{\nabla}

\newcommand{\Rmnum}[1]{\expandafter\@slowromancap\romannumeral #1@}

\allowdisplaybreaks

\numberwithin{equation}{section}
\begin{document}
 \title{On anisotropic energy conservation criteria   of  incompressible fluids}
 \author{Yanqing Wang\footnote{Corresponding author. School of Mathematics and   Information Science, Zhengzhou University of Light Industry, Zhengzhou, Henan  450002,  P. R. China Email: wangyanqing20056@gmail.com},\;~~Wei Wei\footnote{ School of Mathematics and Center for Nonlinear Studies,  Northwest University, Xi'an, Shaanxi 710127,  P. R. China  Email: ww5998198@126.com }
  \;~~and ~~Yulin Ye\footnote{School of Mathematics and Statistics, Henan University, Kaifeng, Henan  475004, P. R. China. Email: ylye@vip.henu.edu.cn }
   }
\date{}
\maketitle
\begin{abstract}
 In this paper, by means of divergence-free condition, we establish
an anisotropic energy conservation class enabling one component of velocity in the largest space $L^{3} (0,T; B^{1/3}_{3,\infty})$ for the 3D inviscid incompressible fluids, which extends the celebrated result obtained by  Cheskidov,    Constantin,   Friedlander and   Shvydkoy in \cite[Nonlinearity 21 (2008)]{[CCFS]}. For   viscous flows, we generalize famous Lions's energy conservation criteria to allow   the
horizontal components and vertical part of velocity to have different integrability.

 \end{abstract}
\section{Introduction}
\label{intro}
\setcounter{section}{1}\setcounter{equation}{0}
The following Euler equations provide a description of
the evolution of  inviscid incompressible fluids
\begin{equation}\left\{\begin{aligned}\label{Euler}
&u_{t}+u\cdot \nabla u+\nabla \Pi=0,\\ &\text{div}\, u=0,\\
&u|_{t=0}=u_{0}(x),
\end{aligned}\right.\end{equation}
where $u$ represents velocity field and $\Pi$ stands for the pressure.
 The
initial  velocity $u_0$ satisfies   $\text{div}\,u_0=0$.

Formally, any smooth solution  of the Euler equations \eqref{Euler} satisfies the kinetic energy conservation
\begin{equation}\label{ei}
\f12\int_{\Omega}|u(x,t)|^{2}dx=\f12\int_{\Omega}|u_{0}(x)|^{2}dx,
\end{equation}
where $\Omega$ is either the whole space $\mathbb{R}^{3}$ or the periodic domain $\mathbb{T}^{3}$.
Lower regularity of solutions to the Euler equations may violate the energy conservation
\eqref{ei}, which was conjectured by Onsager in \cite{[Onsager]}. Indeed, Onsager
conjectured that  the   threshold of solutions in H\"older spaces for the validity of the conservation of kinetic energy is 1/3 in the Euler equations \eqref{Euler}. The rigid side  of Onsager's conjecture  was
originated from Eyink's work \cite{[Eyink]}. Eyink proved that weak solutions conserve the   kinetic energy if $u\in C_{\ast}^{\alpha}$ with $ \alpha>1/3$, where $C_{\ast}^{\alpha}$ is a subspaces of $C^{\alpha}$.
  Subsequently, Constantin, E and Titi \cite{[CET]} proved the   positive  part  of Onsager's conjecture and showed that weak solutions preserve the kinetic  energy \eqref{ei} if $u\in L^{3}(0,T; B^{\alpha}_{3,\infty}(\mathbb{T}^{3}))$ with $\alpha>1/3$. In  \cite{[DR]}, Duchon and Robert introduced the   inertial anomalous dissipation term and applied it to show that the  kinetic  energy is invariant if weak solutions satisfy
  $\int_{\mathbb{T}^{n}}| u(x+\xi,t)-u(x,t) |^3dx \leq C(t)|\xi|\sigma(|\xi|),$ where $C(t)\in L^1(0,T)$ and $\sigma(a)\rightarrow 0$ as $a\rightarrow 0$.
  The results of \cite{[CET],[DR],[Eyink]} were
sharpened  by Cheskidov, Constantin, Friedlander and Shvydkoy  in \cite{[CCFS]}, where they proved that energy is conserved
for weak solutions  in $L^{3} (0,T; B^{1/3}_{3,c(\mathbb{N})}).$
Recently, Fjordholm and Wiedemann refined the aforementioned  results
by a slightly weaker assumption that $u\in L^{3} (0,T;\underline{B}^{1/3 }_{3,VMO})$.
In addition, Berselli \cite{[B]} showed that  the weak solutions in $ L^{\f1\alpha+\delta}(0,T;C^{\alpha}(\mathbb{T}^{3}))$ for $\alpha\in(1/3,1)$ and $\delta>0$   conserve  the energy in the Euler equations.
Very recently, Berselli and Georgiadis \cite{[BG]} extended Constantin, E and Titi's  energy class $L^{3}(0,T; B^{\alpha}_{3,\infty}(\mathbb{T}^{3}))$ to a wider range of exponents as $L^{\f1\alpha}(0,T; B^{\beta}_{\f{2}{1-\alpha},\infty}(\mathbb{T}^{3})), \alpha\in(1/3,1),1/3<\alpha<\beta<1.$  In this direction, one may refer to \cite{[BTW],[BCS],[BG],[BKR],[WWWY],[Chae],[FGSW],[CY]}.
We turn our attention to the other side of Onsager conjecture.
After a series of works \cite{[DS0],[DS1],[DS2]} due to De Lellis and Sz\'ekelyhidi, the flexible part  of Onsager's conjecture for the 3D Euler equations was  confirmed  by Isett in \cite{[Isett]} (see also the recent paper by Giri and   Radu \cite{[GR]}).

 The inclusion relations of the aforementioned   spaces are that,
for $\alpha >1/3$,
\be\label{includ}
C_{\ast}^{\alpha} \subseteq C^\alpha \subseteq B^\alpha_{3,\infty}\subseteq B^{1/3}_{3,c(\mathbb{N})}\subseteq \underline{B}^{1/3}_{3,VMO}\subseteq B^{1/3}_{3,\infty}.
\ee
It is worth pointing out that
Cheskidov, Constantin, Friedlander and Shvydkoy constructed  a divergence-free vector field with non-vanishing energy flux in the largest space $B^{\frac{1}{3}}_{3,\infty}$ in \cite{[CCFS]}. Hence, it seems that it is not expected to obtain energy conservation class  $L^{3} (0,T; B^{1/3}_{3,\infty}).$
On the other hand, anisotropic regularity criteria based on  partial components of velocity in the Leray-Hopf weak solutions of the  viscous flows \eqref{NS} have attracted considerable attention (see \cite{[KZ],[CT],[CZZ],[HLLZ],[NP],[BC07],[Zheng],[WZZ],[WZ]} and   references therein). However, almost all previous sufficient conditions for energy conservation of weak solutions to the
Euler equations \eqref{Euler} are in the context of isotropic regularity.
A natural  question arises whether partial components of velocity in the largest spaces $L^{3} (0,T; B^{1/3}_{3,\infty}) $ can guarantee the energy conservation \eqref{ei} of weak solutions to   invicid fluids \eqref{Euler}.
 We address this issue and formulate our result as follows:
 \begin{theorem}\label{the1.1}
  Let $u$ be a weak solution of the incompressible   Euler equations \eqref{Euler} on   $\mathbb{R}^{3}$. Suppose that  \be\label{1.4}
u_h\in L^{3} (0,T; B^{\alpha}_{3,c(\mathbb{N})}(\mathbb{R}^3)), u_3\in L^{3} (0,T; B^{\beta}_{3,\infty}(\mathbb{R}^3)),  1/3\leq \alpha <1\ \text{and}\ \beta\geq \frac{1-\alpha}{2}.
 \ee
  Then the  kinetic energy is invariant.
 \end{theorem}
This theorem  extends the famous Cheskidov, Constantin, Friedlander and Shvydkoy's energy preservation class $L^{3} (0,T; B^{1/3}_{3,c(\mathbb{N})})$ in \cite{[CCFS]}. The generalization   is twofold. The first one is that energy conservation sufficient conditions \eqref{1.4} allow   the vertical velocity to be  in the largest space $B^{1/3}_{3,\infty}$.   The
other is that the energy class \eqref{1.4} enables two components and the rest of the velocity to have the    different  Besov  regularities.  To the knowledge of authors, it seems to be the first energy balance criterion for  both  inviscid and viscous incompressible flows in the framework of anisotropic  regularities. The proof of this theorem relies on the following observation:
By means of divergence-free condition, we can reformulate the nonlinear terms
\begin{displaymath}
\left( \begin{array}{ccc}
u_{1}\partial_{1}u_{1} & u_{2}\partial_{2}u_{1} & u_{3}\partial_{3}u_{1} \\
u_{1}\partial_{1}u_{2} & u_{2}\partial_{2}u_{2} & u_{3}\partial_{3}u_{2} \\
u_{1}\partial_{1}u_{3} & u_{2}\partial_{1}u_{3} & u_{3}\partial_{3}u_{3}
\end{array} \right)
\end{displaymath}
as
\begin{displaymath}
\left( \begin{array}{ccc}
u_{1}\partial_{1}u_{1} & u_{2}\partial_{2}u_{1} & u_{3}\partial_{3}u_{1} \\
u_{1}\partial_{1}u_{2} & u_{2}\partial_{2}u_{2} & u_{3}\partial_{3}u_{2} \\
u_{1}\partial_{1}u_{3} & u_{2}\partial_{1}u_{3} & -u_{3}(\partial_{1}u_{1}+\partial_{2}u_{2})
\end{array} \right).
\end{displaymath}
We notice that every term includes at least one component of the horizontal velocity. This helps us to achieve the vertical velocity in the largest spaces $ L^{3} (0,T; B^{\f13}_{3,\infty})$ for the energy conservation \eqref{ei}.
As for the torus case, we have the following result.
\begin{theorem}\label{the1.2} Let $u$ be a weak solution to  the incompressible   Euler equations \eqref{Euler} on $\mathbb{T}^{3}$.
Assume that $$u_{h}\in L^{3} (0,T;\underline{B}^{\alpha }_{3,VMO}
 (\mathbb{T}^3) ), u_{3}\in L^{3} (0,T;B^{\beta}_{3,\infty}
   (\mathbb{T}^3)),  1/3\leq \alpha <1\ \text{and}\ \beta\geq \frac{1-\alpha}{2}. $$
Then the energy is   conserved.
\end{theorem}
For the whole space, as a byproduct of Theorem \ref{the1.1}, we have a more general result.
 \begin{coro}\label{coro1.3} Let $u$ be a weak solution of the incompressible   Euler equations \eqref{Euler} on   $\mathbb{R}^{3}$.  Then energy conservation  \eqref{ei} of weak solutions  is valid  if
$$\ba
 u_1 \in L^{p_{1}}(0,T;B^{\f{1}{p_{1 }}}_{\f{{2}{p_{1}}}{p_{1}-1},c(\mathbb{N})}(\mathbb{R}^{3})),
 u_2 \in L^{p_{2}}(0,T;B^{\f{1}{p_{2}}}_{\f{2p_{2}}{p_{2}-1},c(\mathbb{N})}(\mathbb{R}^{3})),  u_3 \in L^{p_{3}}(0,T;B^{\f{1}{p_{3}}}_{\f{{2}{p_{3}}}{p_{3}-1},\infty}(\mathbb{R}^{3})),  \ea$$
 where $p_{1},  p_{2},  p_{3} \in(1,3].$
 \end{coro}
As is said above, even for   viscous fluids \eqref{NS},
it seems that there is no relevant  anisotropic sufficient conditions for energy equality of weak solutions. Next, we shall focus on the validity of energy balance law
\be\label{EENS}
 \f12\|u(T)\|_{L^{2}(\mathbb{R}^{3})}^{2}+  \int_{0}^{T}\|\nabla u\|_{L^{2}(\mathbb{R}^{3})}^{2}ds=
  \f12\|u_0\|_{L^{2}(\mathbb{R}^{3})}^{2}
 \ee
for weak solutions of the following Navier-Stokes equations
 \be\left\{\ba\label{NS}
&u_{t} -\Delta  u+ u\cdot\ti
u  +\nabla \Pi=0, \\
&\Div u=0,\\
&u|_{t=0}=u_0.
\ea\right.\ee
It is well-known that  the   global Leray-Hopf weak solutions of the 3D Navier-Stokes equations \eqref{NS} just meet
 the energy inequality
 \be\label{EINS}
 \f12\|u(T)\|_{L^{2}(\mathbb{R}^{3})}^{2}+  \int_{0}^{T}\|\nabla u\|_{L^{2}(\mathbb{R}^{3})}^{2}ds\leq \f12\|u_0\|_{L^{2}(\mathbb{R}^{3})}^{2}
 \ee
 rather than energy equality \eqref{EENS}. A natural question is what is the minimal regularity for Leray-Hopf weak solutions keeping the energy equality \eqref{EENS}.
 In a seminal work, Lions showed that $u\in L^{4}(0,T;L^{4}(\mathbb{R}^{3})) $ guarantees energy balance law  \eqref{EENS} in \cite{[Lions]}. From then on, the sufficient conditions for energy equality of the Leray-Hopf weak solutions
have attracted considerable attention (see e.g. \cite{[Lions],[Shinbrot],[BY],[CCFS],[CL],[BC],[Zhang],[Galdi]} and references therein). We list some known representative results in this direction below:
A Leray-Hopf weak solution $u$ to the Navier-Stokes equations \eqref{NS} satisfies the energy equality if one of the following conditions holds
\begin{itemize}
\item Shinbrot \cite{[Shinbrot]}:
\be\label{Shinb}
u\in L^{p}(0,T;L^{q}(\mathbb{R}^{3})),~\text{with}~\f{2}{p}+
 \f{2}{q}=1~\text{and}~q\geq 4;\ee
\item Taniuchi \cite{[Taniuchi]}, Beirao da Veiga-Yang \cite{[BY]}:
\be\label{tby}
u\in L^{p}(0,T;L^{q}(\mathbb{R}^{3})),~\text{with}~  \f{1}{p}+
 \f{3}{q}=1 ~and~ 3<q< 4;\ee
\item  Cheskidov-Constantin-Friedlander-Shvydkoy \cite{[CCFS]}: $u\in L^{3}(0,T;B^{\f13}_{3,\infty}(\mathbb{R}^{3}));$
\item Cheskidov-Luo \cite{[CL]}:   $u\in L^{\beta,\infty}(0,T;B^{\f2\beta+\f2p-1}_{p,\infty}(\mathbb{R}^{3})),$~with~$\f2p+\f1\beta<1$~and~$1\leq\beta<p\leq\infty;$
 \item Berselli-Chiodaroli  \cite{[BC]}, Zhang \cite{[Zhang]}:
\be\label{bcz}
\nabla u \in L^{p}\left(0, T ; L^{q}\left(\mathbb{R}^{3}\right)\right),~\text{with}~
\frac{1}{p}+\frac{3 }{q}=2~\text{and}~\frac{3 }{2}<q<\frac{9}{5},~\text{or}~
\frac{1}{p}+\frac{6}{5 q}=1~\text{and}~ q\geq\frac{9}{5}.\ee
   \end{itemize}
In the following, we state our anisotropic criteria for energy conservation of the Leray-Hopf weak solutions.
   \begin{theorem}\label{the1.4}
The energy equality of Leray-Hopf weak solutions $u$ to the 3D Navier-Stokes equations \eqref{NS} is valid if one of the following four conditions is satisfied
 \begin{enumerate}[(1)]
 \item  $u_{h}\in L^{p_{1}}(0,T;L^{q_{1}}(\mathbb{R}^{3}))\cap L^{4}(0,T;L^{4}(\mathbb{R}^{3}))$       and  $u_{3}\in L^{p_{2}}(0,T;L^{q_{2}}(\mathbb{R}^{3}))$  with
$\f{1}{p_{1}}+\f{1}{p_{2}}=\f12$  and  $\f{1}{q_{1}}+\f{1}{q_{2}}=\f12$;
 \item $u_{h}\in L^{p }(0,T;L^{q }(\mathbb{R}^{3}))$ for $\f{2}{p }+\f{3}{q }=1$, $3\leq q < \infty;$
     \item $\nabla u_{h}\in L^{p }(0,T;L^{q }(\mathbb{R}^{3}))$ for $\f{2}{p }+\f{3}{q }=2$, $\f32\leq q <\infty;$
     \item
     $u_{h}\in L^{3}(0,T;B^{\alpha}_{3,\infty}(\mathbb{R}^{3})),$ $u_{3}\in L^{3}(0,T;B^{\beta}_{3,\infty}(\mathbb{R}^{3})),$ $1/3\leq\alpha\leq 1/2,\  \beta\geq \frac{1-\alpha}{2}.$
\end{enumerate}
 \end{theorem}

 \begin{remark}
 A special case of   $u_{h}\in L^{p_{1}}(0,T;L^{q_{1}}(\mathbb{R}^{3}))$       and  $u_{3}\in L^{p_{2}}(0,T;L^{q_{2}}(\mathbb{R}^{3}))$ with  $p_{1}=p_{2}=q_{1}=q_{2}=4$
 reduces to the   famous Lions's
 energy balance class.
 The first result enables the  horizontal part and vertical direction of velocity to have different integrability.
 \end{remark}

 \begin{remark}
The first anisotropic criterion in this theorem is very complicated at first sight, thus it seems that it allows us to  relax Shinbrot energy equality  class \eqref{Shinb} and    \eqref{tby} due to  Beirao da Veiga-Yang (see the following Corollary \ref{coro1.5}).
 \end{remark}

\begin{remark}
The sufficient condition based on
 horizontal   components of velocity for energy balance law lies in the well-known Serrin class. Notice that
the full regularity for Leray-Hopf weak solutions satisfying  $u_h\in L^{\infty}(0, T; L^3(\mathbb{R}^{3}))$ is still open, hence,
 a  potential interesting case in terms of horizontal velocity  for energy equality  is $u_h\in L^{\infty}(0, T; L^3(\mathbb{R}^{3}))$.
\end{remark}

It is well-known that the Lebesgue dominated convergence theorem breaks down for Lebesgue space $L^{\infty}$. To overcome this difficulty, we invoke the Constantin-E-Titi type identity to deal with  the limiting cases $   L^{\infty}(0, T; L^3(\mathbb{R}^{3}))$. For
future work, it would be interesting to prove energy equality class $u_h\in L^{2}(0, T; L^\infty(\mathbb{R}^{3}))$.  When one considers  energy equality condition via $\nabla u_{h}$, it seems that
the term
$$\int_{0}^{T}\int_{\mathbb{R}^{3}} [S_{N}(u_{h}u_{3} )-S_{N} u_{h}S_{N}u_{3} ]\partial_{h} S_{N}u_{3} dxds$$
corresponds to the borderline case of \eqref{1.4} for $\alpha=1$ and $\beta=0$, which brings more difficulty. Here our starting point is the deduction of
$u_{3}\in L^{p}(0,T;B^{0}_{q,1}(\mathbb{R}^{3}))  $
via  low-high frequency techniques (see Lemma \ref{lem2.6new}), which allows us to revisit the proof of Theorem \ref{the1.1} to
prove this energy balance criterion in terms of horizontal gradients. It seems that this observation is of independent interest.
Theorem \ref{the1.4}  has the following   consequence.
\begin{coro}\label{coro1.5}Let $u$ be a Leray-Hopf weak solution to the 3D Navier-Stokes equations \eqref{NS}
in the sense of distributions.
Then for any $0\leq t\leq T$, the energy   of weak solutions   is preserved provided that
\begin{enumerate}[(1)]
  \item $u_h \in L^{p }(0,T;L^{q }(\mathbb{R}^{3})),u_3 \in L^{\f{2p }{p -2}}(0,T;L^{\f{2q }{q -2}}(\mathbb{R}^{3})),$~for$ ~\f{2}{p }+
 \f{2}{q }=1~\text{with} ~q \geq 4$;
 \item $u_h \in L^{p }(0,T;L^{q }(\mathbb{R}^{3})),u_3 \in L^{\f{2p }{p -2}}(0,T;L^{\f{2q }{q -2}}(\mathbb{R}^{3}))$~for~$\f{1}{p }+
 \f{3}{q }=1~$with $~q \leq 4$;
 \item  for $i,j\in\{1,2,3\}$ and $i\neq j$, $u_i \in L^{p_{i}}(0,T;L^{q_{i}}(\mathbb{R}^{3}))  ~\text{with}~\f{2}{p_{i}}+
 \f{2}{q_{i}}=1~\text{for}~q_{i}\geq 4$ and $u_j \in L^{p_{j}}(0,T;L^{q_{j}}(\mathbb{R}^{3})),~\text{with}~\f{1}{p_{j}}+
 \f{3}{q_{j}}=1~\text{for}~q_{j}\leq 4$.
 \end{enumerate}
\end{coro}
We present some comments on Theorem \ref{the1.4} and  Corollary \ref{coro1.5}.
To illustrate our contribution, let  the horizontal and vertical axes be $\f1q$ and $\f1p$  on the coordinate plane,   respectively.
It is known that the region  of  validity of energy identity of Leray-Hopf weak solutions to the tri-dimensional Navier-Stokes equations  is $I$ in Figure \ref{figure1}-\ref{figure4}. Indeed, Shinbrot's,  Beirao da Veiga-Yang's results and H\"older inequality guarantee that the whole  region $I$ means
Lions's energy equality class  $ u \in L^{4}(0,T;L^{4}(\mathbb{R}^{3}))$ (see \eqref{gn1} and \eqref{gn2}).
  The first energy equality criteria in this corollary show  that Shinbrot's energy identity class $u\in L^{\f{8}{3}}(0,T;L^{  8 } (\mathbb{R}^{3}))$
can be replaced by
$u_h\in L^{\f{8}{3}}(0,T;L^{  8 }(\mathbb{R}^{3}) )$ and  $u_3\in L^{ 8}(0,T;L^{\f{8}{3} } (\mathbb{R}^{3}))$.
We would like to point out that the corresponding point of the space $
L^{ 8}(0,T;L^{\f{8}{3} } (\mathbb{R}^{3}))$ is not in region  $I$ and is closer to the natural energy $\f{1}{p}+
 \f{3}{q}=\f32$  of weak solutions than $L^{\f{8}{3}}(0,T;L^{  8 } (\mathbb{R}^{3}))$ in Figure \ref{figure2}.
 The second  energy equality sufficient condition  in  this corollary   asserts that  $u\in   L^{ 7}(0,T;L^{\f{7}{2}} (\mathbb{R}^{3}))$ as a special case of  \eqref{tby} becomes $u_h\in   L^{ 7}(0,T;L^{\f{7}{2} } (\mathbb{R}^{3}))$ and
$u_3\in   L^{ \f{14}{5}}(0,T;L^{\f{14}{3} } (\mathbb{R}^{3}))$. It should be pointed out that
$   L^{ \f{14}{5}}(0,T;L^{\f{14}{3} } (\mathbb{R}^{3}))$ satisfies $\f{1}{p }+
 \f{3}{q }=1$ but $q>4$ (see Figure \ref{figure2}).
%
 The same  scenario also occurs for all the points in region  $I$ via Theorem \ref{the1.4} in  Figure \ref{figure3}.
 Notice that the midpoint of  $(\f{1}{q},\f{1}{p})$ and $(\f{q-2}{2q},\f{p-2}{2p})$ is  always $(\f{1}{4},\f{1}{4})$, therefore, the region $I$ of  horizontal components  and $II$ of  vertical velocity  are symmetric with
 respect to the Lions's famous result in  Figure \ref{figure3}, where the region $II$ is determined  by the three lines $\f2p+\f3q=1,  \f{2}{p}+
 \f{2}{q}=1$, and $\f{1}{p}+
 \f{3}{q}=\f32$ in this figure. When the horizontal component of the velocity approaches the Serrin class, the vertical velocity is close to the natural energy of weak solutions (see Figure \ref{figure2} and \ref{figure3}).
Notice that the classical Shinbrot's and Beirao da Veiga-Yang's results \eqref{Shinb}-\eqref{tby} are
in terms of all components of velocity. Inspired by Corollary \ref{coro1.3}, the new ingredient of the third energy equality sufficient conditions in this corollary is to allow partial components of velocity to be in Shinbrot's class with different
integrability and the rest components to be in Beirao da Veiga-Yang's class.
 Two special cases of the third energy equality sufficient conditions (see Figure \ref{figure4}) read
\be\ba &u_1 \in L^{ \f52}(0,T;L^{10} (\mathbb{R}^{3})) ,
 u_2\in L^{ 3}(0,T;L^{6} (\mathbb{R}^{3})),  u_3 \in L^{ \f{10}{3}}(0,T;L^{5} (\mathbb{R}^{3}));\\
 &u_1 \in L^{ \f{14}{5}}(0,T;L^{7} (\mathbb{R}^{3})),
 u_2\in L^{\f{100}{19}}(0,T;L^{\f{100}{27} } (\mathbb{R}^{3})),  u_3 \in L^{ 7}(0,T;L^{\f72} (\mathbb{R}^{3})).
\ea\ee
Indeed, choosing arbitrary three points  $(\f{1}{q_{1}},\f{1}{p_{1}})$, $(\f{1}{q_{2}},\f{1}{p_{2}})$   and  $(\f{1}{q_{3}},\f{1}{p_{3}})$   in region I in Figure \ref{figure4}, we can immediately obtain a sufficient condition for weak solutions keeping energy conservation
\be\ba
u_1 \in L^{ p_{1}}(0,T;L^{q_{1}} (\mathbb{R}^{3})),
 u_2\in L^{ p_{2}}(0,T;L^{q_{2}} (\mathbb{R}^{3})),  u_3 \in L^{ p_{3}}(0,T;L^{q_{3}} (\mathbb{R}^{3})),
\ea\ee
 which is a generalization of Shinbrot's and  Beirao da Veiga-Yang's energy balance classes \eqref{Shinb}-\eqref{tby}.

\begin{figure}[htbp]
  \centering
  \begin{minipage}{0.48\textwidth}
    \centering

\begin{tikzpicture}[scale=9,>= stealth]
\draw[->] (0,0)--(0.7,0)node[right]{\tiny{$\f1q$}};
\draw[->](0,0)--(0,0.6)node[right]{\tiny{$\f1p$}};
\draw (0,0.5)--(0.33333,0)node[pos=0.5,below,sloped]{\tiny Serrin};
\draw (0,0.5)--(0.25,0.25)node[pos=0.5,above,sloped]{\tiny Shinbrot};
\draw (0.25,0.25)--(0.33333,0)  ;

\draw (0.5,0)--(0.166666,0.5)node[pos=0.7,above,sloped]{\tiny weak solutions};
\draw[fill=black] (0.25,0.25) circle(.002);
\node[right] at ( 0.001, 0.012) {\tiny{0}};
\draw[<-](0.2857,0.142857)--(0.35,0.142857)node[right]{\tiny Beirao da Veiga-Yang-Taniuchi};
\draw[<-](0.25,0.25)--(0.4,0.25)node[right]{\tiny Lions};
\node[below] at (0.5,0.01) {\tiny $( \infty,2)$};
\node[below] at (0.333333,0.01) {\tiny $( \infty,3)$};
\node[above] at (0.1666666,0.5) {\tiny $( 2,6)$};
\node[above] at (0,0.5) {\tiny $( 2,\infty)$};
\node[above] at (0.26,0.25) {\tiny$(4,4)$};
\node[above] at (0.22,0.22) {\tiny$I$};
\end{tikzpicture}
    \caption{classical results}
    \label{figure1}
  \end{minipage}
  \hfill
  \begin{minipage}{0.48\textwidth}
    \centering

\begin{tikzpicture}[scale=9,>= stealth]
\draw[->] (0,0)--(0.7,0)node[right]{\tiny{$\f1q$}};
\draw[->](0,0)--(0,0.6)node[right]{\tiny{$\f1p$}};
\draw (0,0.5)--(0.33333,0);
\draw (0,0.5)--(0.25,0.25);
\draw (0.25,0.25)--(0.33333,0)  ;

\draw (0.5,0)--(0.166666,0.5);
\draw[fill=black] (0.25,0.25) circle(.002);
\node[right] at ( 0.001, 0.012) {\tiny{0}};

 \draw[style=dashed] (0.125,0.375)--(0.48,0.495) ;
 \draw[style=dashed] ( 0.375,0.125)--(0.52,0.49);
 \node[below] at (0.5,0.53)  {\tiny $(u_h,u_3)$};
 \node[left] at (0.125,0.375)  {\tiny $(\f83,8)$};
 \node[right] at ( 0.375,0.125)  {\tiny $( 8,\f83)$};
\draw[fill=black] ( 0.375,0.125) circle(.002);
\draw[fill=black] ( 0.125,0.375) circle(.002);

\draw[style=dashed] (0.2857,0.142857)--(0.16,0.09) ;
\draw[style=dashed] ( 0.2143,0.357143)--(0.108,0.09);
\node[below] at (0.14,0.11)  {\tiny $( u_3,u_h)$};

\node[right] at (0.2857,0.142857)  {\tiny $(7,\f72)$};
\node[above] at ( 0.2143,0.357143)  {\tiny $(\f{14}{5},\f{14}{3})$};

\draw[fill=black] (0.2857,0.142857) circle(.002);
\draw[fill=black] ( 0.2143,0.357143)  circle(.002);

\draw[style=dashed] (0.25,0.2)--(0.58,0.2) ;
\draw[style=dashed] ( 0.25,0.3)--(0.64,0.21);
\draw[fill=black] (0.25,0.2) circle(.002);

\node[right] at (0.57,0.2)  {\tiny $(u_h,u_3)$};
\node[left] at (0.26,0.2)  {\tiny $(5,4)$};
\node[right] at ( 0.25,0.33333)  {\tiny $(\f{10}{3},4)$};
\draw[fill=black] ( 0.25,0.3) circle(.002);

\node[below] at (0.5,0.01) {\tiny $( \infty,2)$};
\node[below] at (0.333333,0.01) {\tiny $( \infty,3)$};
\node[above] at (0.1666666,0.5) {\tiny $( 2,6)$};
\node[above] at (0,0.5) {\tiny $( 2,\infty)$};
\node[right] at (0.244,0.25) {\tiny$(4,4)$};
\node[above] at (0.22,0.22) {\tiny$I$};
\end{tikzpicture}
    \caption{special case of new energy class}
    \label{figure2}
  \end{minipage}
\end{figure}
 \begin{figure}[h]
  \centering
  \begin{minipage}{0.48\textwidth}
    \centering

\begin{tikzpicture}[scale=9,>= stealth]
\fill[gray!30] (0.25,0.25) -- ( 0.5,0)-- ( 0.1666667,0.5)--(0.25,0.25);

\draw[->] (0,0)--(0.7,0)node[right]{\tiny{$\f1q$}};
\draw[->](0,0)--(0,0.6)node[right]{\tiny{$\f1p$}};
\draw (0,0.5)--(0.33333,0);
\draw (0,0.5)--(0.25,0.25);
\draw (0.25,0.25)--(0.33333,0)  ;
 \draw[style=dashed] (0.25,0.25)--(0.5,0);
 \draw[style=dashed] (0.25,0.25)--(0.166666,0.5);

\draw (0.5,0)--(0.166666,0.5);
\draw[fill=black] (0.25,0.25) circle(.002);
\node[left] at (0.21,0.22) {\tiny $(p,q)$};
\node[right] at (0.205,0.205) {\tiny $u_h$};
\draw[fill=black] (0.21,0.212) circle(.002);


\node[above] at (0.28,0.42) {\tiny $( u_h,u_3)$};

\draw[fill=black] (0.29,0.288) circle(.002);

\node[right] at (0.29,0.288) {\tiny $(\f{2p}{p-2},\f{2q}{q-2})$};
\node[left] at (0.3,0.29) {\tiny $u_3$};
\node[right] at ( 0.001, 0.012) {\tiny{0}};
\node[below] at (0.5,0.01) {\tiny $( \infty,2)$};
\node[below] at (0.333333,0.01) {\tiny $( \infty,3)$};
\node[above] at (0.1666666,0.48) {\tiny $( 2,6)$};
\node[above] at (0,0.48) {\tiny $( 2,\infty)$};
\node[right] at (0.244,0.25) {\tiny$(4,4)$};
\draw[fill=white] (0,0.5) circle(.004);
\node[above] at (0.25,0.13 ) {\tiny$I$};
\node[above] at (0.26,0.31) {\tiny$II$};
\end{tikzpicture}
    \caption{new anisotropic energy equality criteria}
    \label{figure3}
  \end{minipage}
  \hfill
  \begin{minipage}{0.48\textwidth}
    \centering

\begin{tikzpicture}[scale=9,>= stealth]
\draw[->] (0,0)--(0.7,0)node[right]{\tiny{$\f1q$}};
\draw[->](0,0)--(0,0.6)node[right]{\tiny{$\f1p$}};
\draw (0,0.5)--(0.33333,0);
\draw (0,0.5)--(0.25,0.25);
\draw (0.25,0.25)--(0.33333,0)  ;

\draw (0.5,0)--(0.166666,0.5);
\draw[fill=black] (0.25,0.25) circle(.002);
\node[right] at ( 0.001, 0.012) {\tiny{0}};
\node[below] at (0.5,0.01) {\tiny $( \infty,2)$};
\node[below] at (0.333333,0.01) {\tiny $( \infty,3)$};
\node[above] at (0.1666666,0.5) {\tiny $( 2,6)$};
\node[above] at (0,0.5) {\tiny $( 2,\infty)$};
\node[above] at (0.26,0.25) {\tiny$(4,4)$};

\draw[style=dashed] (0.1,0.4)--(0.05,0.15) ;

\draw[style=dashed] (0.1666667,0.333333)--(0.1 ,0.16) ;
\draw[style=dashed] (0.2,0.3)--(0.15,0.15) ;

\node[above] at (0.1 ,0.12) {\tiny$(u_1,u_2,u_3)$};

\draw[style=dashed] (0.1428571428,0.357142857)--(0.4,0.35) ;

\draw[style=dashed] (0.27,0.19)--(0.4537,0.3286) ;
\draw[style=dashed] (0.2857,0.142857)--(0.5,0.332);

\node[above] at (0.47,0.32) {\tiny$(u_1,u_2,u_3)$};
\node[above] at (0.22,0.22) {\tiny$I$};
\end{tikzpicture}
    \caption{anisotropic classical results}
    \label{figure4}
  \end{minipage}
\end{figure}

The rest of this  paper is organized as follows. In Section 2,
we present the notations and some basic  materials of Besov spaces.
Two critical lemmas for the proof of Theorem \ref{the1.4} are also given. Section 3 is devoted to the anisotropic criteria for energy conservation of weak solutions to inviscid fluids.
In Section 4, we consider anisotropic sufficient conditions for the energy equality of viscous fluids.

\section{Notations and  key auxiliary lemmas} \label{section2}

Throughout this paper, we will use the summation convention on repeated indices. $C$ will denote positive absolute constants which may be different from line to line unless otherwise stated in this paper.  For $p\in [1,\,\infty]$, the notation $L^{p}(0,\,T;X)$ stands for the set of measurable functions $f$ on the interval $(0,\,T)$ with values in $X$ and $\|f\|_{X}$ belonging to $L^{p}(0,\,T)$.

Denote by
$\mathcal{S}(\mathbb{R}^{n})$ the Schwartz space of rapidly
decreasing smooth functions, $\mathcal{S}'(\mathbb{R}^{n})$ the
space of tempered distributions,
$\mathcal{S}'(\mathbb{R}^{n})/\mathcal{P}(\mathbb{R}^{n})$ the
quotient space of tempered distributions which modulo polynomials.
We use $\mathcal{F}f$ or $\widehat{f}$ to denote the Fourier transform of a tempered distribution $f$, and $\mathcal{F}^{-1}f$ represents the inverse Fourier transform of $f$.
   To define Besov  spaces, we need the following dyadic unity partition
(see e.g. \cite{[BCD]}). Choose two nonnegative radial
functions $\varrho$, $\varphi\in C^{\infty}(\mathbb{R}^{n})$
supported respectively in the ball $\{\xi\in
\mathbb{R}^{n}:|\xi|\leq \frac{4}{3} \}$ and the shell $\{\xi\in
\mathbb{R}^{n}: \frac{3}{4}\leq |\xi|\leq
  \frac{8}{3} \}$ such that
\begin{equation*}
 \varrho(\xi)+\sum_{j\geq 0}\varphi(2^{-j}\xi)=1, \quad
 \forall\xi\in\mathbb{R}^{n}; \qquad
 \sum_{j\in \mathbb{Z}}\varphi(2^{-j}\xi)=1, \quad \forall\xi\neq 0.
\end{equation*}
It follows that $\varphi(\xi)=\varrho(\xi/2)-\varrho(\xi)$ for all $\xi\in
\mathbb{R}^{n}$. Denote $h=\mathcal{F}^{-1} \varphi $ and $\tilde{h}=\mathcal{F}^{-1}\varrho$, then inhomogeneous dyadic blocks  $\Delta_{j}$ are defined by
$$
\Delta_{j} u:=0 ~~ \text{if} ~~ j \leq-2, ~~ \Delta_{-1} u:=\varrho(D) u =\int_{\mathbb{R}^n}\tilde{h}(y)u(x-y)dy,$$
$$\text{and}~~\Delta_{j} u:=\varphi\left(2^{-j} D\right) u=2^{jn}\int_{\mathbb{R}^n}h(2^{j}y)u(x-y)dy  ~~\text{if}~~ j \geq 0.
$$
And the inhomogeneous low-frequency cut-off operator $S_j$ for any $j \geq 0$ is defined by
$$
S_{j}u:= \sum_{k\leq j-1}\Delta_{k}u=\varrho(2^{-j}D)u.$$
The homogeneous dyadic blocks $\dot{\Delta}_{j}$ are  defined  for every $j\in\mathbb{Z}$ by
\begin{equation*}
  \dot{\Delta}_{j}u:= \varphi(2^{-j}D)u.
\end{equation*}
Then for $-\infty <\alpha<\infty$ and $1\leq p,q\leq \infty,$ the homogeneous Besov semi-norm $ \|f\|_{\dot{B}^{\alpha}_{p, q}(\mathbb{R}^{n})}$ of $f\in \mathcal{S}'(\mathbb{R}^{n})/\mathcal{P}(\mathbb{R}^{n})$ is given by
\begin{equation*}
	\begin{aligned}
  \norm{f}_{\dot{B}^{\alpha}_{p, q}(\mathbb{R}^{n})}:=\left\{\begin{array}{lll}\left(\sum_{j\in \mathbb{Z}}2^{jq\alpha}\norm{\dot{\Delta}_{j} f} _{L^p(\mathbb{R}^{n})}^q\right)^{1/q},~~\text{if}\ q\in [1,\infty),\\
  	\sup_{j\in \mathbb{Z}}2^{j\alpha}\norm {\dot{\Delta}_{j}f} _{L^p(\mathbb{R}^{n})},~~~~~~~~~~\text{if}~q=\infty.\end{array}\right.
\end{aligned}\end{equation*}
Moreover, we define the inhomogeneous Besov norm $\norm{f}_{B^\alpha_{p,q}(\mathbb{R}^{n})}$ of $f\in \mathcal{S}^{'}(\mathbb{R}^{n})$ as
\begin{equation*}
	\begin{aligned}
  \norm{f}_{B^{\alpha}_{p, q}(\mathbb{R}^{n})}:=\left\{\begin{array}{lll}\left(\sum_{j\in \mathbb{Z}}2^{jq\alpha}\norm{\Delta_{j} f} _{L^p(\mathbb{R}^{n})}^q\right)^{1/q},~~\text{if}\ q\in [1,\infty),\\
  	\sup_{j\in \mathbb{Z}}2^{j\alpha}\norm {\Delta_{j}f} _{L^p(\mathbb{R}^{n})},~~~~~~~~~\text{if}~q=\infty.\end{array}\right.
\end{aligned}\end{equation*}
Then we denote the homogeneous Besov space by
\begin{equation*}
  \dot{B}^{\alpha}_{p,q}(\mathbb{R}^{n}):=\Big\{f\in\mathcal{S}'(\mathbb{R}^{n})/\mathcal{P}(\mathbb{R}^{n})\B|
  \norm{f}_{\dot{B}^{\alpha}_{p,q}(\mathbb{R}^{n})}<\infty  \Big\},
\end{equation*} and the inhomogeneous Besov space by
\begin{equation*}
  B^{\alpha}_{p,q}(\mathbb{R}^{n}):=\Big\{f\in\mathcal{S}'(\mathbb{R}^{n})\B|\norm{f}_{B^{\alpha}_{p,q}(\mathbb{R}^{n})}<\infty  \Big\}.
\end{equation*}
In the spirit of  \cite{[CCFS]}, we define $ {B}^\alpha _{p,c(\mathbb{N})}$ to be the class of all tempered distributions $f$ for which
\begin{equation}\label{2.1}
\norm{f}_{ {B}^\alpha _{p,\infty}}<\infty~ \text{and}~ 	\lim_{j\rightarrow \infty} 2^{j\alpha}\norm{ {\Delta}_j f}_{L^p}=0,~~\text{for any}~1\leq p\leq \infty.
\end{equation}
In addition, one can also define the homogeneous Besov space with positive indices  in terms of finite differences.
  For the convenience of readers, we give the detail on periodic domain.   For $1\leq q\leq \infty$ and $0<\alpha<1$, the homogeneous Besov space $\dot{B}^{\alpha}_{q,\infty}(\mathbb{T}^{n})$ is the space of functions $f$ on the $d$ dimensional torus $\mathbb{T}^{n}=[0,1]^{n}$ equipped with the semi-norm
 \be\label{besov1} \|f\|_{\dot{B}^{\alpha}_{q,\infty}(\mathbb{T}^{n})}= \sup_{y\in \mathbb{T}^{n}} |y|^{-\alpha}\B\| f(x+y)-f(x)\B\|
 _{L_{x}^{q}(\mathbb{T}^{n}\cap(\mathbb{T}^{n}-y))} <\infty,\ee
where $\mathbb{T}^{n}-y=\{x-y|x\in \mathbb{T}^{n}\},$ and the inhomogeneous Besov space $B^{\alpha}_{q,\infty}(\mathbb{T}^{n})$ is the set of functions $f\in L^{q}(\mathbb{T}^{n})$ equipped with the norm
 $$\ba
  \|f\|_{B^{\alpha}_{q,\infty}(\mathbb{T}^{n})}= \|f\|_{L^{q}(\mathbb{T}^{n})}
  +\|f\|_{\dot{B}^{\alpha}_{q,\infty}(\mathbb{T}^{n})}<\infty.\ea$$
A function $f$ belongs to
the Besov-VMO space  $L^p(0,T;\underline{B}^{\alpha}_{q,VMO}(\mathbb{T}^n))$  if it satisfies
$$\|f\|_{L^p(0,T;L^q(\mathbb{T}^n))}<\infty,$$
and
$$\ba
&\lim_{\varepsilon\rightarrow0}\f{1}{\varepsilon^{\alpha}}\left(\int_0^T\B[\int_{\mathbb{T}^n} \fbxo|f(x)-f(y)|^{q}dydx \B]^{\f{p}{q}}dt\right)^{\f1p}\\
=&\lim_{\varepsilon\rightarrow0}\f{1}{\varepsilon^{\alpha}}\left(\int_0^T\B[\int_{\mathbb{T}^n} \fbxoo|f(x)-f(x-y)|^{q}dydx \B]^{\f{p}{q}}dt\right)^{\f{1}{p}}=0.
\ea$$

{\bf Mollifier kernel:} Let $\eta_{\varepsilon}:\mathbb{R}^{n}\rightarrow \mathbb{R}$ be a standard mollifier, i.e. $\eta(x)=C_0e^{-\frac{1}{1-|x|^2}}$ for $|x|<1$ and $\eta(x)=0$ for $|x|\geq 1$, where $C_0$ is a constant such that $\int_{\mathbb{R}^n}\eta (x) dx=1$. For $\varepsilon>0$, we define the rescaled mollifier $\eta_{\varepsilon}(x)=\frac{1}{\varepsilon^n}\eta(\frac{x}{\varepsilon})$ and for  any function $f\in L^1_{loc}(\mathbb{R}^n)$, its mollified version is defined as
$$f^\varepsilon(x)=(f*\eta_{\varepsilon})(x)=\int_{\mathbb{R}^n}f(x-y)\eta_{\varepsilon}(y)dy,\ \ x\in \mathbb{R}^n.$$
Next, we collect  some lemmas  which will be used in the present paper.
\begin{lemma}(\cite{[Chae],[Yu]})\label{lem2.2}
 Suppose that  $f\in L^p(0,T; {B}^\alpha_{q,\infty}(\mathbb{T}^{n}))$, $g\in L^p(0,T;{B}^\beta_{q,c(\mathbb{N})}(\mathbb{T}^{n}))$ with $\alpha, \beta\in (0,1)$,  $ p,q\in [1,\infty]$,   then there holds that,  for any $k\in \mathbb{N}^+$, as $\varepsilon\rightarrow0,$
 \begin{enumerate}[(1)]
 \item $ \|f^{\varepsilon} -f \|_{L^{p}(0,T;L^{q}(\mathbb{T}^{n}))}\leq \text{O}(\varepsilon^{\alpha})\|f\|_{L^p(0,T; {B}^\alpha_{q,\infty}(\mathbb{T}^n))} $;
   \item   $ \|\nabla^{k}f^{\varepsilon}  \|_{L^{p}(0,T;L^{q}(\mathbb{T}^{n}))}\leq \text{O}(\varepsilon^{\alpha-k})\|f\|_{L^p(0,T; {B}^\alpha_{q,\infty}(\mathbb{T}^{n}))} $;
       \item $ \|g^{\varepsilon} -g \|_{L^{p}(0,T;L^{q}(\mathbb{T}^{n}))}\leq \text{o}(\varepsilon^{\beta})\|g\|_{L^p(0,T; {B}^\beta_{q,c(\mathbb{N})}(\mathbb{T}^{n}))} $;
   \item   $ \|\nabla^{k}g^{\varepsilon}  \|_{L^{p}(0,T;L^{q}(\mathbb{T}^{n}))}\leq \text{o}(\varepsilon^{\beta-k})\|g
       \|_{L^p(0,T; {B}^\beta_{q,c(\mathbb{N})}(\mathbb{T}^{n}))}$.
 \end{enumerate}
\end{lemma}
\begin{lemma}(\cite{[YWL]})	\label{lem2.3}
	Assume that $0<\alpha,\beta<1$, $1\leq p,q,p_{1},p_{2}\leq\infty$ and $\frac{1}{p}=\frac{1}{p_1}+\frac{1}{p_2}$.
Then as $\varepsilon\to 0$, there holds
	\begin{align} \label{cet}
		\|(fg)^{\varepsilon}- f^{\varepsilon}g^{\varepsilon}\|_{L^p(0,T;L^q(\mathbb{T}^3))} \leq  \text{o}(\varepsilon^{\alpha+\beta}),
	\end{align}
provided that one of the following two conditions is satisfied,
\begin{enumerate}[(1)]
\item  $f\in L^{p_1}(0,T;\underline{B}^{\alpha}_{q_1,VMO} (\mathbb{T}^3) )$, $g\in L^{p_2}(0,T;\underline{B}^{\beta}_{q_2,VMO}(\mathbb{T}^3)  )$, $1\leq q_{1},q_{2}\leq\infty$, $\frac{1}{q}=\frac{1}{q_1}+\frac{1}{q_2}$;
 \item  $f\in L^{p_1}(0,T;\underline{B}^{\alpha}_{q_{1},VMO}(\mathbb{T}^3))$, $g\in L^{p_2}(0,T;\dot{B}^{\beta}_{q_{2},\infty} (\mathbb{T}^3))$, $1\leq q_{1},q_{2}\leq\infty$, $\frac{1}{q}=\frac{1}{q_1}+\frac{1}{q_2},q_{2}\geq \f{q_{1}}{q_{1}-1}$ and $p_{2}\geq\f{q_{1}}{q_{1}-1}$.
 \end{enumerate}\end{lemma}
Next, we prove two lemmas for the study of energy equality of weak solutions in the Navier-Stokes equations.
\begin{lemma}\label{lem2.1}
	Let $ p, q, p_2, q_2\in[1,+\infty)$ and $p_1, q_1\in[1,+\infty]$  with
$\frac{1}{p}=\frac{1}{p_1}+\frac{1}{p_2},\frac{1}{q}=\frac{1}{q_1}+\frac{1}{q_2} $. Assume $f\in L^{p_1}(0,T;L^{q_1}(\mathbb{R}^3)) $ and $g\in
L^{p_2}(0,T;L^{q_2}(\mathbb{R}^3))$, then   there holds
	\begin{equation}\label{a4}
	\lim_{N\rightarrow\infty}\|S_{N}(f g)  -S_{N}f  S_{N}g \|_{L^p(0,T;L^q(\mathbb{R}^3))}= 0.
	\end{equation}

\end{lemma}
\begin{proof}
The proof of this lemma   is folk for the case that $ p, q, p_1, q_1, p_2, q_2\in[1,+\infty)$. It suffices to notice that
$$S_{N}(fg) -S_{N}f S_{N}g=
S_{N}(fg) -fg+fg-  f S_{N}g+ f S_{N}g - S_{N}f S_{N}g.$$
For case $p_1, q_1\in[1,+\infty]$, we will apply Constantin-E-Titi type  identity to prove this lemma. Indeed, in view of  the fact that $2^{ nN}\int_{\mathbb{R}^n}\tilde{h}(2^{N}y)dy=\mathcal{F}(\tilde{h})(0)=\varrho(0)=1,$ we notice  that
\be\ba\label{cetilem2.3}
&S_{N}(fg) -S_{N}f S_{N}g\\
=&  2^{ 3N}\int_{\mathbb{R}^3}\tilde{h}(2^{N}y)[f(x-y)-f(x)]
[g(x-y)-g(x)]dy-(f-S_{N}f)(g-S_{N}g).
\ea\ee
In view of changing of variable, we reformulate  \eqref{cetilem2.3} as
 \be\ba\label{ceti2}
&S_{N}(fg) -S_{N}f S_{N}g\\
=&   \int_{\mathbb{R}^3}\tilde{h}(z)\B[f(x-\f{z}{2^{N}})-f(x)\B]
\B[g(x-\f{z}{2^{N}})-g(x)\B]dz-(f-S_{N}f)(g-S_{N}g).
\ea\ee
It follows from the H\"older  inequality and triangle inequality
  that
\be\ba\label{2.8}
&\|
S_{N}(fg) -S_{N}f S_{N}g\|_{L^{p }(0,T;L^{q}(\mathbb{R}^{3}))}\\
\leq&
  \int_{\mathbb{R}^3}\tilde{h}(z)\B\|f(x-\f{z}{2^{N}})-f(x)\B\|_{L^{p_{1}}
  (0,T;L^{q_{1}}(\mathbb{R}^{3}))}
\B\|g(x-\f{z}{2^{N}})-g(x)\B\|_{L^{p_{2}}
  (0,T;L^{q_{2}}(\mathbb{R}^{3}))}dz\\&+\|f-S_{N}f\|_{L^{p_{1}}
  (0,T;L^{q_{1}}(\mathbb{R}^{3}))}\| g-S_{N}g \|_{L^{p_{2}}
  (0,T;L^{q_{2}}(\mathbb{R}^{3}))}\\
  \leq&
  C\|f\|_{L^{p_{1}}
  (0,T;L^{q_{1}}(\mathbb{R}^{3}))}\int_{\mathbb{R}^3}\tilde{h}(z)
\B\|g(x-\f{z}{2^{N}})-g(x)\B \|_{L^{p_{2}}
  (0,T;L^{q_{2}}(\mathbb{R}^{3}))}dz\\&+C\|f\|_{L^{p_{1}}
  (0,T;L^{q_{1}}(\mathbb{R}^{3}))}\|g-S_{N}g\|_{L^{p_{2}}
  (0,T;L^{q_{2}}(\mathbb{R}^{3}))}.
 \ea\ee
Since $p_2,q_2\in[1, \infty)$,   we deduce from the following approximations of the identity that
\be\label{2.9}
\lim_{N\rightarrow\infty}\|g-S_{N}g\|_{L^{p_{2}}
  (0,T;L^{q_{2}}(\mathbb{R}^{3}))}=0.
  \ee
The strong-continuity of translation operators on Lebesgue spaces allows us to write,
for
 $1\leq p_{2},q_{2}<\infty$,
\be\label{2.10}
 \lim_{N\rightarrow\infty}\B\|g(x-\f{z}{2^{N}})-g(x)\B\|_{L^{p_{2}}
  (0,T;L^{q_{2}}(\mathbb{R}^{3}))}=0.
  \ee
We conclude by the Lebesgue dominated convergence theorem and \eqref{2.10}  that
 \be\label{2.11}
 \lim_{N\rightarrow\infty} \int_{\mathbb{R}^3}\tilde{h}(z)
\B\|g(x-\f{z}{2^{N}})-g(x)\B \|_{L^{p_{2}}
  (0,T;L^{q_{2}}(\mathbb{R}^{3}))}dz=0.
   \ee
Substituting \eqref{2.9} and \eqref{2.11} into \eqref{2.8}, we finish the proof of this lemma.
\end{proof}
It is well-known that Leray-Hopf weak solutions of the tridimensional  Navier-Stokes equations satisfy $u\in L^{p}(0,T;L^{q}(\mathbb{R}^{3}))$ for $\f{2}{p}+\f{3}{q}=\f32 $  by Gagliardo-Nirenberg inequality. By means of low-high frequency
techniques, we will show that Leray-Hopf weak solutions are in
$L^{p}(0,T;B^{0}_{q,1}(\mathbb{R}^{3}))  $ for $\f{2}{p}+\f{3}{q}=\f32 $ with $2<q<6$, which plays an important role in the proof of Theorem \ref{the1.4}.
\begin{lemma}\label{lem2.6new}
Suppose that
$f\in L^{\infty}(0,T;L^{2}(\mathbb{R}^{3}))  $  and $\nabla f\in L^{2}(0,T;L^{2}(\mathbb{R}^{3}))  $.
Then there holds, for $2<q<6$,
$f\in L^{p}(0,T;B^{0}_{q,1}(\mathbb{R}^{3}))  $ and
$f\in L^{p}(0,T;\dot{B}^{0}_{q,1}(\mathbb{R}^{3}))   $ with $\f{2}{p}+\f{3}{q}=\f32.$
\end{lemma}
\begin{proof}
According to the definition of Besov spaces, we write
\be\ba\label{2.13}
\|f\|_{B^{0}_{q,1}}=\sum^{N-1}_{j=-1} \|\Delta_{j}f\|_{L^{q}(\mathbb{R}^{3})}
+\sum^{\infty}_{j=N} \|\Delta_{j}f\|_{L^{q}(\mathbb{R}^{3})}
\ea\ee
By Bernstein inequality, we see that
 \be\ba\label{2.14}
\|\Delta_{j}f\|_{L^{q}}\leq C 2^{j[3(\f12-\f1q)]}\|\Delta_{j}f\|_{L^{2}(\mathbb{R}^{3})}\leq C 2^{j[3(\f12-\f1q)]}\| f\|_{L^{2}(\mathbb{R}^{3})}
\ea\ee
It follows from interpolation and Sobolev embedding theorem that
 \be\ba\label{2.15}
\|\Delta_{j}f\|_{L^{q}(\mathbb{R}^{3})}\leq & \|\Delta_{j}f\|_{L^{2}(\mathbb{R}^{3})}^{\f{6-q}{2q}}\|\Delta_{j}f\|^{\f{3q-6}{2q}}_{L^{6}(\mathbb{R}^{3})}\\
\leq & C \|\Delta_{j}f\|_{L^{2}(\mathbb{R}^{3})}^{\f{6-q}{2q}}\|\nabla\Delta_{j}f\|^{\f{3q-6}{2q}}_{L^{2}(\mathbb{R}^{3})}\\
\leq &C 2^{ j[\f{(3q-6)}{2q}-1]}2^{j}\|\Delta_{j}f\|_{L^{2}(\mathbb{R}^{3})}\\
\leq &C 2^{ j[\f{(3q-6)}{2q}-1]} \|\nabla f\|_{L^{2}(\mathbb{R}^{3})},
\ea\ee
where the Bernstein inequality and $N>-1$ was used.

Plugging \eqref{2.14} and \eqref{2.15} into \eqref{2.13}, we conclude by
the  straightforward  calculation that
\be\ba\label{2.16new}
\|f\|_{B^{0}_{q,1}}\leq& C\| f\|_{L^{2}(\mathbb{R}^{3})}\sum^{N-1}_{j=-1}  2^{j[3(\f12-\f1q)]}
+C \|\nabla f\|_{L^{2}(\mathbb{R}^{3})}\sum^{\infty}_{j=N}2^{\f{j(q-6)}{2q}}
\\
\leq &C   2^{N[3(\f12-\f1q)]}\| f\|_{L^{2}(\mathbb{R}^{3})}
+C  2^{\f{N(q-6)}{2q}}\|\nabla f\|_{L^{2}(\mathbb{R}^{3})}
\ea\ee
where we have used $2<q<6$.

Before going further, we write
$$
2^{N[3(\f12-\f1q)]}\| f\|_{L^{2}(\mathbb{R}^{3})}
\approx  2^{\f{N(q-6)}{2q}}\|\nabla f\|_{L^{2}(\mathbb{R}^{3})},$$
which enables us to rewrite \label{2.16} as
\be\ba
\|f\|_{B^{0}_{q,1}}\leq C\| f\|_{L^{2}(\mathbb{R}^{3})}^{\f{6-q}{2q}}\|\nabla f\|^{\f{3q-6}{2q}}_{L^{2}(\mathbb{R}^{3})}.
\ea\ee
By a time integration, we arrive at
\be\ba
\int_{0}^{T}\|f\|^{ \f{4q}{3q-6 }}_{B^{0}_{q,1}}dt\leq C\| f\|_{L^{\infty}(0,T; L^{2}(\mathbb{R}^{3}))}^{\f{6-q}{2q}}\int_{0}^{T}\|\nabla f\|^{2}_{L^{2}(\mathbb{R}^{3})}
dt.\ea\ee
This completes the proof of this lemma.
\end{proof}

 \section{Energy conservation   of weak solutions of ideal fluids}

This section is concerned with the energy conservation of the Euler equations \eqref{Euler}. Making use of divergence-free condition, we shall  revisit how much regularities are required for   weak solutions to preserve energy in the Euler equations and establish an energy conservation class allowing one component  of velocity to be in the largest space $B_{3, \infty}^{\f13}$. Firstly, we deal with the whole space case via Littlewood-Paley theory as follows.
\subsection{Whole space case}
\begin{proof}[Proof of Theorem \ref{the1.1}]
Applying the inhomogeneous Littlewood-Paley operator $S_{N}$ to the Euler equations,
we get
$$
S_{N}u_{t}  + S_{N}(u\cdot\ti
u ) +\nabla S_{N}\Pi=0,  ~~
 \Div S_{N}u=0.$$
It follows from the energy estimate, divergence-free condition and
 integration by parts that
 \be\ba
 \f12\|S_{N}u(T)\|_{L^{2}(\mathbb{R}^{3})}^{2} = \f12\|S_{N}u_{0}\|_{L^{2}(\mathbb{R}^{3})}^{2}+\sum_{ij}^{3}\int_{0}^{T}\int_{\mathbb{R}^{3}}(S_{N}(u_{i}u_{j} )\partial_{i} S_{N}u_{j})dxds.
\ea\ee
We compute
\be\ba\label{3.2}
&\sum_{ij}^{3} \int_{\mathbb{R}^{3}}(S_{N}(u_{i}u_{j} )\partial_{i} S_{N}u_{j})dx \\
=&\sum_{ i}^{2}\sum_{ j}^{3} \int_{\mathbb{R}^{3}}(S_{N}(u_{i}u_{j} )\partial_{i} S_{N}u_{j})dx + \sum_{ j}^{3} \int_{\mathbb{R}^{3}}(S_{N}(u_{3}u_{j} )\partial_{3} S_{N}u_{j})dx
\\
=& \sum_{ i}^{2}\sum_{ j}^{2} \int_{\mathbb{R}^{3}}(S_{N}(u_{i}u_{j} )\partial_{i} S_{N}u_{j})dx +\sum_{ id}^{2} \int_{\mathbb{R}^{3}}(S_{N}(u_{i}u_{3} )\partial_{i} S_{N}u_{3})dxds\\&+ \sum_{ j}^{2} \int_{\mathbb{R}^{3}}(S_{N}(u_{3}u_{j} )\partial_{3} S_{N}u_{j})dx +  \int_{\mathbb{R}^{3}}(S_{N}(u_{3}u_{3} )\partial_{3} S_{N}u_{3})dx \\
=&I+II+III+IV.
\ea\ee
In view of the incompressible condition, we see that
$$
\sum_{i,j=1}^{3}\int_{\mathbb{R}^{3}} S_{N}u_{j} \partial_{j}S_{N} u_{i}
S_{N}u_{i} dx=0,
$$
which allows us to write
 \be\ba
 I=  \int_{\mathbb{R}^{3}} [S_{N}(u_{h}u_{h} )-S_{N} u_{h}S_{N}u_{h} ]\partial_{h} S_{N}u_{h} dx, \\
 II=   \int_{\mathbb{R}^{3}} [S_{N}(u_{h}u_{3} )-S_{N} u_{h}S_{N}u_{3} ]\partial_{h} S_{N}u_{3} dx, \\
 III=  \int_{\mathbb{R}^{3}} [S_{N}(u_{3}u_{h} )-S_{N} u_{3}S_{N} u_{h} ]\partial_{3} S_{N}u_{h} dx,
\ea\ee
and
 \be\ba\label{3.4}
IV=&- \int_{\mathbb{R}^{3}} [S_{N}(u_{3}u_{3} )-S_{N} u_{3}S_{N}u_{3}  ] (\partial_{1} S_{N}u_{1}+\partial_{2}S_{N}u_{2})dx \\
=&- \int_{\mathbb{R}^{3}} [S_{N}(u_{3}u_{3} )-S_{N} u_{3}S_{N}u_{3}  ] \partial_{1} S_{N}u_{1} dx \\&-\int_{\mathbb{R}^{3}} [S_{N}(u_{3}u_{3} )-S_{N} u_{3}S_{N}u_{3}  ]  \partial_{2}S_{N}u_{2} dx.
\ea\ee
As a consequence, we have
\be\ba\label{3.5}
 \f12\|S_{N}u(T)\|_{L^{2}(\mathbb{R}^{3})}^{2}  =&\f12\|S_{N}u_{0}\|_{L^{2}(\mathbb{R}^{3})}^{2}+\int_{0}^{T}(I+II+III+IV)ds.
\ea\ee
Therefore, we conclude by the H\"older inequality
  that
\begin{equation}\label{key03}
	\begin{aligned}
	  |I |
=&\B| \int_{\mathbb{R}^{3}}\big[S_{N}(u_{h}u_{h}) -S_{N}u_{h} S_{N}u_{h} \big]\partial_{h} S_{N}u_{h}dx \B|\\
\leq	& \|S_{N}(u_{h}u_{h})  -S_{N}u_{h}S_{N}u_{h}\|_{ L^{\f{3}{2}} (\mathbb{R}^{3})}  \|\partial_{h}S_{N}u_{h}\|_{ L^{3}  (\mathbb{R}^{3})}.
\end{aligned}\end{equation}
According to the fact that $2^{ 3N}\int_{\mathbb{R}^3}\tilde{h}(2^{N}y)dy=\mathcal{F}(\tilde{h})(0)=\varrho(0)=1,$ we obtain the following Constantin-E-Titi type identity
\be\ba\label{ceti}
&S_{N}(u_{h}u_{h}) -S_{N}u_{h} S_{N}u_{h}\\
=&  2^{ 3N}\int_{\mathbb{R}^3}\tilde{h}(2^{N}y)[u_{h}(x-y)-u_{h}(x)]
[u_{h}(x-y)-u_{h}(x)]dy-(u_{h}-S_{N}u_{h})(u_{h}-S_{N}u_{h}).
\ea\ee
By virtue of  the Minkowski inequality, we arrive at
$$\ba\label{key}
 &\|S_{N}(u_{h}u_{h})  -S_{N}u_{h}S_{N}u_{h}\|_{  L^{\f{3}{2}}  (\mathbb{R}^{3})}
 \\\leq& 2^{ 3N}\int_{\mathbb{R}^3}|\tilde{h}(2^{N}y)|\| u_{h}(\,\cdot-y)-u_{h}(\cdot) \|_{ L^{3}  (\mathbb{R}^{3})}
 \| u_{h}(\,\cdot-y)-u_{h}(\cdot) \|_{ L^{3} (\mathbb{R}^{3}) }dy
\\&+\| u_{h}-S_{N}u_{h} \|_{ L^{3}  (\mathbb{R}^{3})} \| u_{h}-S_{N}u_{h}\|_{L^{3} (\mathbb{R}^{3})}\\
:=&\,I_{1}+I_{2}.
\ea$$
The Newton-Leibniz formula together with the Minkowski inequality and the Bernstein inequality ensures that
\begin{align}
&\|u_{h}(\,\cdot-y)-u_{h}(\cdot)\|_{ L^{3} (\mathbb{R}^{3})}\nonumber\\
\leq& C
\sum_{k< N}\|\Delta_{k}u_{h} (\,\cdot-y)-\Delta_{k}u_{h} (\cdot)\|_{L^{3} (\mathbb{R}^{3})} +\sum_{k\geq N}\|\Delta_{k}u_{h} (\,\cdot-y)-\Delta_{k}u_{h}(\cdot)\|_{L^{3} (\mathbb{R}^{3})}\nonumber\\
\leq&C\sum_{k< N}\B\| -\int_{0}^{1}y\cdot\nabla{\Delta}_{k}u_{h}(\,\cdot-\vartheta y)d\vartheta \B\|_{L^{3} (\mathbb{R}^{3})}+\sum_{k\geq N}\| {\Delta}_{k}u_{h}\|_{L^{3} (\mathbb{R}^{3})}\nonumber\\
\leq&C\sum_{k< N} \int_{0}^{1}|y|\B\|\nabla{\Delta}_{k}u_{h}(\,\cdot-\vartheta y)\B\|_{L^{3} (\mathbb{R}^{3})}d\vartheta +\sum_{k\geq N}\| {\Delta}_{k}u_{h}\|_{L^{3} (\mathbb{R}^{3})}\nonumber\\
=&C \sum_{k< N} |y| \B\|\nabla{\Delta}_{k}u_{h}\B\|_{L^{3} (\mathbb{R}^{3})} +\sum_{k\geq N}\| {\Delta}_{k}u_{h}\| _{L^{3} (\mathbb{R}^{3})} \nonumber\\
 \leq&C \sum_{k< N}2^{ k}|y| \| {\Delta}_{k}u_{h}\| _{L^{3} (\mathbb{R}^{3})}+\sum_{k\geq N}\| {\Delta}_{k}u_{h}\| _{L^{3} (\mathbb{R}^{3})} \nonumber\\
\leq&C  2^{ N(1-\alpha)}|y| \sum_{k< N}2^{- (N-k)(1-\alpha)}2^{ k\alpha}\| {\Delta}_{k}u_{h}\|_{L^{3} (\mathbb{R}^{3})} +2^{- \alpha N}\sum_{k\geq N} 2^{ (N-k)\alpha} 2^{ k\alpha}\| {\Delta}_{k}u_{h}\| _{L^{3} (\mathbb{R}^{3})},\label{0e3.3}
\end{align}
where $\alpha\in(0,1)$ is a constant to be determined later.

Before going further, in the spirit of  \cite{[CCFS]}, we set the following localized kernel
\be\label{K1}\Gamma_{1}(j)=\left\{\begin{aligned}
	&2^{j \alpha },~~~~~~~~~\text{if}~j\leq0;\\
	& 2^{-(1-\alpha)j},~~\text{if}~j>0,
\end{aligned}\right.
\ee
and  we denote $ {d}_{hj}=2^{j\alpha}\| {\Delta}_{j}u_{h}\|_{L^{3} (\mathbb{R}^{2})}$.

As a consequence, we get
$$\ba
\| u_{h}(x-y)-u_{h}(x) \|_{ L^{3} (\mathbb{R}^{2}) }\leq& C\left[ 2^{N(1-\alpha)}|y| +2^{  -\alpha  N}\right]\left(\Gamma_{1}\ast {d}_{hj}\right)(N)\\
\leq & C( 2^{N}|y| +1)2^{ -\alpha N}\left(\Gamma_{1}\ast  {d}_{hj}\right)(N).
 \ea$$

Since
$$ \sup_{N\in\mathbb{Z}} 2^{ 3N}\int_{\mathbb{R}^3}  |\tilde{h}(2^{N}y)|(2^{N}|y| +1)^{2}dy<\infty,$$

Hence, we deduce from  \eqref{0e3.3} and  \eqref{K1} that
$$I_{1}\leq C 2^{ -2\alpha N}\left(\Gamma_{1}\ast {d}_{hj}\right)^{2}(N).$$

On the other hand, we notice that
$$\ba
\| u_{h}-{S}_{N}u_{h} \|_{ L^{3}  (\mathbb{R}^{3})}\leq& C   \sum_{k\geq N }\| \Delta_{k} u_{h}  \|_{L^{3}(\mathbb{R}^{3})}
\\ \leq
&C
2^{- \alpha N}\sum_{k\geq N} 2^{ (N-k)\alpha} 2^{ k\alpha}\| {\Delta}_{k}u_{h}\| _{L^{3} (\mathbb{R}^{3})} \\ \leq
&C   2^{ - \alpha  N}\left(\Gamma_{1}\ast   {d}_{hj} \right)(N) .
\ea$$
As a consequence, we have
$$I_{2} \leq
C2^{-2\alpha N}\left(\Gamma_{1}\ast   {d}_{hj} \right)^{2}(N).
$$
Hence, there holds
\be\ba\label{3.10}
 \|S_{N}(u_{h}u_{h})  -S_{N}u_{h}S_{N}u_{h}\|_{  L^{\f{3}{2}}  (\mathbb{R}^{3})}
 \leq C 2^{ -2\alpha N}\left(\Gamma_{1}\ast  {d}_{hj}\right)^{2}(N).
 \ea\ee
The Bernstein inequality leads to
 \be\label{3.8}\ba
\|\partial_{j}S_{N}u_{h}\|_{ L^{3}  (\mathbb{R}^{3})} \leq& C  \sum_{k< N }\| \nabla\Delta_{k} u_{h}  \|_{L^{3}(\mathbb{R}^{3})}
\\
\leq
& C    \sum_{k< N }2^{ k}\| \Delta_{k} u_{h}  \|_{L^{3}(\mathbb{R}^{3})}
\\ =&C 2^{(1- \alpha )N}
\sum_{k< N}2^{- (N-k)(1-\alpha)}2^{ k\alpha}\| {\Delta}_{k}u_{h}\|_{L^{3} (\mathbb{R}^{3})} \\ \leq&C 2^{(1- \alpha )N}
\Gamma_{1}\ast   {d}_{hj}.
\ea\ee
Inserting   \eqref{3.5}  and  \eqref{3.8} into \eqref{key03}, we end up with
\be\ba\label{add}
  |I|
\leq C 2^{(1- 3 \alpha)N} \B(
\Gamma_{1}\ast   {d}_{hj}\B)^{3}.
\ea\ee
To proceed further, we denote
\be\label{K2}
\Gamma_{2}(j)=\left\{\begin{aligned}
	&2^{j\beta},~~~~~~~~\text{if}~j\leq0;\\
	& 2^{-(1-\beta)j},~~\text{if}~j>0,
\end{aligned}\right.
\ee
and    $ {d}_{vj}=2^{j\beta}\| {\Delta}_{j}u_{3}\|_{L^{3} (\mathbb{R}^{3})}$, where $0<\beta<1$.\\
By a suitable modification of the deduction  of \eqref{add}, we can show
that
\be\ba\label{add2}
  |II|+|IV|
\leq C 2^{[1-  (\alpha+2\beta)]N} \B(
\Gamma_{1}\ast   {d}_{hj}\B)  \B(
\Gamma_{2}\ast   {d}_{vj}\B)^{2},
\ea\ee and
\be\ba\label{add3}
  |III|
\leq C 2^{[1-  (2\alpha+\beta)]N} \B(
\Gamma_{1}\ast   {d}_{hj}\B)^{2}  \B(
\Gamma_{2}\ast   {d}_{vj}\B).
\ea\ee
It is worth pointing out that $\alpha\geq1/3$ and $ \alpha+2\beta\geq1 $ leads to
 $ 2\alpha+\beta\geq1$. Hence, we derive from \eqref{add}, \eqref{add2} and \eqref{add3} that
\be\ba
&\left|\int_{0}^{T}(I+II+III+IV)ds\right| \\
\leq& C\int_{0}^{T}\B(
\Gamma_{1}\ast   {d}_{hj}\B)^{3}ds+C\int_{0}^{T}\B(
\Gamma_{1}\ast   {d}_{hj}\B)  \B(
\Gamma_{2}\ast   {d}_{vj}\B)^{2}ds+C\int_{0}^{T}\B(
\Gamma_{1}\ast   {d}_{hj}\B)^{2}  \B(
\Gamma_{2}\ast   {d}_{vj}\B)ds
\\\leq& C\int_{0}^{T}\B(
\Gamma_{1}\ast   {d}_{hj}\B)^{3}ds
+C\B( \int_{0}^{T} (
\Gamma_{1}\ast   {d}_{hj} )^{3}ds\B)^{\f13}\B[\int_{0}^{T} (
\Gamma_{2}\ast   {d}_{vj} )^{3}ds \B]^{\f23}
\\&+C\B[ \int_{0}^{T} (
\Gamma_{1}\ast   {d}_{vj} )^{3}ds\B]^{\f23}
\B[ \int_{0}^{T} (
\Gamma_{2}\ast   {d}_{hj} )^{3}ds\B]^{\f13},
 \ea\ee
where the H\"older inequality was used. \\
This implies that
 \be\ba
&\left|\int_{0}^{T}(I+II+III+IV)ds\right| \\
\leq& C\int_{0}^{T}\|u_{h}\|_{B_{3,\infty}^{\alpha}}^{3}ds
+C\B(\int_{0}^{T}\|u_{h}\|_{B_{3,\infty}^{\alpha}}^{3}ds\B)^{\f13}\B(\int_{0}^{T}\|u_{ 3}\|_{B_{3,\infty}^{\beta}}^{3}ds\B)^{\f23}
\\&+C\B(\int_{0}^{T}\|u_{h}\|_{B_{3,\infty}^{\alpha}}^{3}ds\B)^{\f23}
\B(\int_{0}^{T}\|u_{ 3}\|_{B_{3,\infty}^{\beta}}^{3}ds\B)^{\f13}.
 \ea\ee
Then we conclude by the
dominated convergence theorem  that
$$\ba
&\limsup_{N\rightarrow +\infty}\left|\int_{0}^{T}(I+II+III+IV)ds\right|=0.
 \ea$$
 Hence, we get the energy balance law \eqref{ei} via passing to the limit of $N$.
  \end{proof}
Secondly, we will study the  periodic case by
a Constantin-E-Titi type commutator estimate \eqref{cet} involving
Besov-VMO spaces and
mollifier kernel.
\subsection{Torus case}
\begin{proof}[Proof of Theorem \ref{the1.2}]
Let us begin by
mollifying the equations \eqref{Euler}  in spatial direction to get
\be\ba\label{rmhd}
&\partial_{t}{u^{\varepsilon}} +   (u\cdot\nabla u)^{\varepsilon} +\nabla\Pi^{\varepsilon}= 0.
\ea\ee
Multiplying  \eqref{rmhd} by $u^{\varepsilon}$ and integrating it with respect to $x$ and $t$, we conclude by incompressible condition and integration by parts that
\begin{equation}\ba\label{3.6}
\f12\|u^{\varepsilon}(T)\|^{2}_{L^{2}(\mathbb{T}^3)}
-\f12\|u^{\varepsilon}_{0}  \|^{2}_{L^{2}(\mathbb{T}^3)} =\sum_{ij}^{3}\int_{0}^{T}\int_{\mathbb{T}^{3}}  (u_{j}u_{i} )^{\varepsilon}\partial_{j} u_{i}^{\varepsilon} dxds.
\ea\end{equation}
By virtue of  divergence-free condition, we get
$$\sum_{i,j=1}^{3}\int_{0}^{T}\int_{\mathbb{T}^{3}}  u_{j}^\varepsilon u_{i} ^{\varepsilon}\partial_{j} u_{i}^{\varepsilon} dxds=0.
$$
 Repeating the calculations \eqref{3.2}-\eqref{3.4}, we deduce
\begin{equation}\label{c27}
	\begin{aligned}
	 \f12\|u^{\varepsilon}(T)\|^{2}_{L^{2}(\mathbb{T}^3)}
	 	-\f12\|u^{\varepsilon}_{0} \|^{2}_{L^{2}(\mathbb{T}^3)}
	 	=&\int_{0}^{T}\int_{\mathbb{T}^{3}} [ (u_{h}u_{h} )^{\varepsilon}-  u_{h}^{\varepsilon}u_{h}^{\varepsilon} ]\partial_{h}  u_{h}^{\varepsilon} dxds, \\
 &+ \int_{0}^{T}\int_{\mathbb{T}^{3}} [ (u_{h}u_{3} )^{\varepsilon}-  u_{h}^{\varepsilon}u_{3}^{\varepsilon} ]\partial_{h}  u_{3}^{\varepsilon} dxds, \\
&+ \int_{0}^{T}\int_{\mathbb{T}^{3}} [ (u_{3}u_{h} )^{\varepsilon}-  u_{3}^{\varepsilon} u_{h}^{\varepsilon} ]\partial_{3} u_{h} ^{\varepsilon}dxds\\
& -\int_{0}^{T}\int_{\mathbb{T}^{3}} [ (u_{3}u_{3} )^{\varepsilon}-  u_{3}^{\varepsilon} u_{3}^{\varepsilon}  ] \partial_{1}  u_{1}^{\varepsilon} dxds \\&-\int_{0}^{T}\int_{\mathbb{T}^{3}} [ (u_{3}u_{3} )^{\varepsilon}-  u_{3}^{\varepsilon}u_{3} ^{\varepsilon} ]  \partial_{2} u_{2}^{\varepsilon} dxds\\
&=:I+II+III+IV+V.
	\end{aligned}
	\end{equation}
Thanks to Lemma \ref{lem2.2},  $u_{h}\in L^{3} (0,T;\underline{B}^{\alpha }_{3,VMO}
  )$ and $ u_{3}\in L^{3} (0,T;B^{\beta}_{3,\infty})$, we get
\be\ba
\|\nabla  u_{h}^{\varepsilon}\|_{L^{3}(0,T;L^{3}(\mathbb{T}^3))}\leq o(\varepsilon^{1-\alpha}),\\
\|\nabla  u_{3}^{\varepsilon}\|_{L^{3}(0,T;L^{3}(\mathbb{T}^3))}\leq O(\varepsilon^{1-\beta}).
\ea\ee
From Lemma \ref{lem2.3}, we conclude by $u_{h}\in L^{3} (0,T;\underline{B}^{\alpha }_{3,VMO}
  )$ and $ u_{3}\in L^{3} (0,T;B^{\beta}_{3,\infty})$ that
\be\ba
\|
 (u_{h}u_{h} )^{\varepsilon}-  u_{h}^{\varepsilon}u_{h}^{\varepsilon}
\|_{L^{\f32}(0,T;L^{\f32}(\mathbb{T}^3))}\leq o(\varepsilon^{2\alpha}),
\ea\ee
and
\be\ba
\|
(u_{h}u_{3} )^{\varepsilon}-  u_{h}^{\varepsilon}u_{3}^{\varepsilon}
\|_{L^{\f32}(0,T;L^{\f32}(\mathbb{T}^3))}\leq o(\varepsilon^{ \alpha+\beta}).
\ea\ee
In the light of H\"older inequality, we remark
\be\ba
&|I|\leq\|
 (u_{h}u_{h} )^{\varepsilon}-  u_{h}^{\varepsilon}u_{h}^{\varepsilon}
\|_{L^{\f32}(0,T;L^{\f32}(\mathbb{T}^3))}\|\nabla  u_{h}^{\varepsilon}\|_{L^{3}(0,T;L^{3}(\mathbb{T}^3))}\leq o(\varepsilon^{ 3\alpha-1}),\\
&|II|\leq \|
(u_{h}u_{3} )^{\varepsilon}-  u_{h}^{\varepsilon}u_{3}^{\varepsilon}
\|_{L^{\f32}(0,T;L^{\f32}(\mathbb{T}^3))}\|\nabla  u_{3}^{\varepsilon}\|_{L^{3}(0,T;L^{3}(\mathbb{T}^3))}\leq o(\varepsilon^{ \alpha+2\beta-1}).
\ea\ee
Likewise,
\be\ba
&|III|\leq o(\varepsilon^{ 2\alpha+\beta-1}),
\\
&|IV|\leq o(\varepsilon^{\alpha+ 2\beta-1}),
\\&|V|\leq o(\varepsilon^{\alpha+2\beta-1}).
\ea\ee
Since $ \alpha\geq1/3$ and $ \alpha+2\beta\geq1$, there holds $2\alpha+\beta\geq1$. As a consequence,  passing to the limit of $\varepsilon$ in \eqref{c27}, we get the energy conservation \eqref{ei}.
\end{proof}
Next, we present an application of Theorem \ref{the1.1} to deduce more general anisotropic energy preservation class.
\begin{proof}[Proof of  Corollary \ref{coro1.3}]
It suffices to show that $ u_{h}\in L^{3}(0,T;
B^{\f13}_{3,c(\mathbb{N})}(\mathbb{R}^{3}))$ and $ u_{3}\in L^{3}(0,T;
B^{\f13}_{3, \infty}(\mathbb{R}^{3}))$ via Theorem \ref{the1.1}. First, we assert   that $u_{1}\in L^{p_{1}}(0,T;
B^{\f{1}{p_{1}}}_{\f{2p_{1}}{p_{1}-1},c(\mathbb{N})}(\mathbb{R}^{3}))$ leads to \be\label{3.252}
u_{1}\in L^{3}(0,T;
B^{\f13}_{3,c(\mathbb{N})}(\mathbb{R}^{3})).
 \ee
 Indeed, we deduce from interpolation inequality in Lebesgue spaces that
$$ \|\Delta_{j}u_{1}\|_{L^{3}(\mathbb{R}^{3})}\leq \|\Delta_{j}u_{1}\|_{L^{2}(\mathbb{R}^{3})}^{1-\f{p_{1}}{3}}
 \|\Delta_{j}u_{1}\|_{L^{\f{2p_{1}}{p_{1}-1}}(\mathbb{R}^{3})}^{\f{p_{1}}{3}},\ \text{for\ any }\ p_{1}\in [1,3],$$
which helps us to get
\be\ba\label{3.262}
 2^{\f13 j}\|\Delta_{j}u_{1}\|_{L^{3}(\mathbb{R}^{3})}\leq& \|\Delta_{j}u_{1}\|_{L^{2}(\mathbb{R}^{d})}^{1-\f{p_{1}}{3}}
\B[2^{j\f{1}{p_1}}\|\Delta_{j}u_{1}\|_{L^{\f{2p_{1}}{p_{1}-1}}(\mathbb{R}^{3})}\B]^{\f{p_{1}}{3}}\\
\leq& C\| u_{1}\|_{L^{2}(\mathbb{R}^{3})}^{1-\f{p_{1}}{3}}  \|u_{1}\|^{\f{p_{1}}{3}}_{
B^{\f{1}{p_{1}}}_{\f{2p_{1}}{p_{1}-1},\infty}(\mathbb{R}^{3})}.
\ea\ee
With the help of the definition of Besov spaces, we further get
$$ \|u_{1}\|_{
B^{\f13}_{3,\infty}(\mathbb{R}^{3})}\leq  C\| u_{1}\|_{L^{2}(\mathbb{R}^{3})}^{1-\f{p_{1}}{3}}  \|u_{1}\|^{\f{p_{1}}{3}}_{
B^{\f{1}{p_{1}}}_{\f{2p_{1}}{p_{1}-1},\infty}(\mathbb{R}^{3})}.$$
By the time integration, we infer that
$$ \|u_{1}\|_{L^{3}(0,T;
B^{\f13}_{3,\infty}(\mathbb{R}^{3}))}\leq  C\| u_{1}\|_{L^{\infty}(0,T;L^{2}(\mathbb{R}^{3}))}^{1-\f{p_{1}}{3}}  \|u_{1}\|^{\f{p_{1}}{3}}_{L^{p_{1}}(0,T;
B^{\f1p_{1}}_{\f{2p_{1}}{p_{1}-1},\infty}(\mathbb{R}^{3}))}.$$
In addition, we deduce  from \eqref{3.262} that
$$\ba 2^{\f13 j}\|\Delta_{j}u_{1}\|_{L^{3}(\mathbb{R}^{3})}\leq& C\| u_{1}\|_{L^{2}(\mathbb{R}^{3})}^{1-\f{p_{1}}{3}}
\B[2^{j\f{1}{p_{1}}}\|\Delta_{j}u_{1}\|_{L^{\f{2p_{1}}{p_{1}-1}}(\mathbb{R}^{3})}\B]^{\f{p}{3}},
\ea$$
where $C$ is independent of $j$.
Hence, the assertion \eqref{3.252} is verified.
In a similar
manner, we conclude by $u_{2}\in L^{p_{1}}(0,T;
B^{\f{1}{p_{2}}}_{\f{2p_{2}}{p_{2}-1},c(\mathbb{N})}(\mathbb{R}^{3}))$ and $u_{3}\in L^{p_{3}}(0,T;
B^{\f{1}{p_{3}}}_{\f{2p_{3}}{p_{3}-1},c(\mathbb{N})}(\mathbb{R}^{3}))$ that
$
 u_{2}\in L^{3}(0,T;
B^{\f13}_{3,c(\mathbb{N})}(\mathbb{R}^{3}))$ and $ u_{3}\in L^{3}(0,T;
B^{\f13}_{3, \infty}(\mathbb{R}^{3}))$, respectively. At this stage, we apply Theorem \ref{the1.1} to   complete the proof of this corollary.
\end{proof}

  \section{Energy equality of weak solutions of   viscous incompressible flows}
In the spirit of the previous section, we  study the energy equality of weak solutions to the 3D Navier-Stokes equations \eqref{NS} in the context of anisotropic velocity field.
\begin{proof}[Proof of Theorem \ref{the1.4}]
Along the same line of \eqref{3.5}, we know that
\be\label{4.1}\ba
&\f12\|S_{N}u(T)\|_{L^{2}(\mathbb{R}^{3})}^{2}+\int_{0}^{T}\|\nabla S_{N}u\|_{L^{2}(\mathbb{R}^{3})}^{2}ds \\=&\f12\|S_{N}u_{0}\|_{L^{2}(\mathbb{R}^{3})}^{2}+\int_{0}^{T}(I+II+III+IV+V)ds,
\ea\ee
where
 $$\ba
 &I=  \int_{0}^{T}\int_{\mathbb{R}^{3}} [S_{N}(u_{h}u_{h} )-S_{N} u_{h}S_{N}u_{h} ]\partial_{h} S_{N}u_{h} dxds, \\
& II=   \int_{0}^{T}\int_{\mathbb{R}^{3}} [S_{N}(u_{h}u_{3} )-S_{N} u_{h}S_{N}u_{3} ]\partial_{h} S_{N}u_{3} dxds, \\
 &III= \int_{0}^{T} \int_{\mathbb{R}^{3}} [S_{N}(u_{3}u_{h} )-S_{N} u_{3}S_{N} u_{h} ]\partial_{3} S_{N}u_{h} dxds,
\\
&IV =-\int_{0}^{T} \int_{\mathbb{R}^{3}} [S_{N}(u_{3}u_{3} )-S_{N} u_{3}S_{N}u_{3}  ] \partial_{1} S_{N}u_{1} dx,
 \\
 & V = -\int_{0}^{T}\int_{\mathbb{R}^{3}} [S_{N}(u_{3}u_{3} )-S_{N} u_{3}S_{N}u_{3}  ]  \partial_{2}S_{N}u_{2} dxds.
\ea$$

(1) Case 1: $u_{h}\in L^{p_{1}}(0,T;L^{q_{1}}(\mathbb{R}^{3}))$, $u_{h}\in L^{4}(0,T;L^{4}(\mathbb{R}^{3}))$ and $u_{3}\in L^{p_{2}}(0,T;L^{q_{2}}(\mathbb{R}^{3}))$ with
$\f{1}{p_{1}}+\f{1}{p_{2}}=\f12$ and $\f{1}{q_{1}}+\f{1}{q_{2}}=\f12$, $p_{1}, q_{1}\geq4$.

 In view of the H\"older inequality, we find
$$\ba
|I|\leq& \|S_{N}(u_{h}u_{h} )-S_{N} u_{h}S_{N}u_{h}\|_{L^{2}(0,T;L^{2}(\mathbb{R}^{3}))}\|\partial_{h} S_{N}u_{h}\|_{L^{2}(0,T;L^{2}(\mathbb{R}^{3}))}\\
\leq& C\|S_{N}(u_{h}u_{h} )-S_{N} u_{h}S_{N}u_{h}\|_{L^{2}(0,T;L^{2}(\mathbb{R}^{3}))}\|\partial_{h} u_{h}\|_{L^{2}(0,T;L^{2}(\mathbb{R}^{3}))}.
\ea$$
Likewise,
\be\ba
&|II|\leq C  \|S_{N}(u_{h}u_{3} )-S_{N} u_{h}S_{N}u_{3}\|_{L^{2}(0,T;L^{2}(\mathbb{R}^{3}))}\|\partial_{h} u_{3}\|_{L^{2}(0,T;L^{2}(\mathbb{R}^{3}))},\\
&|III|\leq C\|S_{N}(u_{3}u_{h} )-S_{N} u_{3}S_{N} u_{h}\|_{L^{2}(0,T;L^{2}(\mathbb{R}^{3}))}\|\partial_{3} u_{h}\|_{L^{2}(0,T;L^{2}(\mathbb{R}^{3}))}.
\ea\ee
We derive from Lemma \ref{lem2.1} and $u_{h}\in L^{4}(0,T;L^{4}(\mathbb{R}^{3}))$
that
\be\ba\label{4.4}
\lim_{N\rightarrow\infty}
\|S_{N}(u_{h}u_{h} )-S_{N} u_{h}S_{N}u_{h}\|_{L^{2}(0,T;L^{2}(\mathbb{R}^{3}))}=0,
\ea\ee
which leads to
\be\ba\label{4.5}
\lim_{N\rightarrow\infty}|I|=0.
\ea\ee

Employing Lemma  \ref{lem2.1} once again, we conclude by $u_{h}\in L^{p_{1}}(0,T;L^{q_{1}}(\mathbb{R}^{3}))$ and $u_{3}\in L^{p_{2}}(0,T;L^{q_{2}}(\mathbb{R}^{3}))$  that
\be\ba\label{4.6}
\lim_{N\rightarrow\infty} \|S_{N}(u_{h}u_{3} )-S_{N} u_{h}S_{N}u_{3}\|_{L^{2}(0,T;L^{2}(\mathbb{R}^{3}))}+\|S_{N}(u_{3}u_{h} )-S_{N} u_{3}S_{N} u_{h}\|_{L^{2}(0,T;L^{2}(\mathbb{R}^{3}))}=0.
\ea\ee
where $\f{1}{p_{1}}+\f{1}{p_{2}}=\f12$ and $\f{1}{q_{1}}+\f{1}{q_{2}}=\f12$.

A combination of \eqref{4.4} and \eqref{4.6} enables us to infer that
\be\ba\label{4.7}
\lim_{N\rightarrow\infty}(|II|+|III|)=0.
\ea\ee
According to integration by parts, we rewrite $IV$ and $V$ as
\be\ba
&IV = 2\int_{0}^{T} \int_{\mathbb{R}^{3}} [S_{N}(u_{3}\partial_{1}u_{3} )-S_{N} u_{3}S_{N}\partial_{1}u_{3}  ] S_{N}u_{1} dxds,
\\
 & V =  2\int_{0}^{T}\int_{\mathbb{R}^{3}} [S_{N}(u_{3}\partial_{2}u_{3} )-S_{N} \partial_{2}u_{3}S_{N}u_{3}  ]S_{N}u_{2} dxds.
\ea\ee
The H\"older inequality implies that
\be\ba\label{4.9}
|IV|+|V|\leq& C\|S_{N}(u_{3}\partial_{1}u_{3} )-S_{N} u_{3}S_{N}\partial_{1}u_{3}\|_{L^{\f{p_{1}}{p_{1}-1}}(0,T;L^{\f{q_{1}}{q_{1}-1}}(\mathbb{R}^{3}))}
\|S_{N}u_{h}\|_{L^{p_{1}}(0,T;L^{q_{1}}(\mathbb{R}^{3}))}
\\
\leq& C\|S_{N}(u_{3}\partial_{1}u_{3} )-S_{N} u_{3}S_{N}\partial_{1}u_{3}\|_{L^{\f{p_{1}}{p_{1}-1}}(0,T;L^{\f{q_{1}}{q_{1}-1}}(\mathbb{R}^{3}))}
\| u_{h}\|_{L^{p_{1}}(0,T;L^{q_{1}}(\mathbb{R}^{3}))}.
\ea\ee
It follows from Lemma \ref{lem2.1}, $u_{h}\in L^{p_{1}}(0,T;L^{q_{1}}(\mathbb{R}^{3}))$ and $\partial_{1}u_{3}\in L^{2}(0,T;L^{2}(\mathbb{R}^{3}))$ that
\be\ba\label{4.10}
\lim_{N\rightarrow\infty}
\|S_{N}(u_{3}\partial_{1}u_{3} )-S_{N} u_{3}S_{N}\partial_{1}u_{3}\|_{L^{\f{p_{1}}{p_{1}-1}}(0,T;L^{\f{q_{1}}{q_{1}-1}}(\mathbb{R}^{3}))}
=0,
\ea\ee
where
$\f{1}{p_{1}}+\f{1}{p_{2}}=\f12$ and $\f{1}{q_{1}}+\f{1}{q_{2}}=\f12.$

Combining  \eqref{4.9} and \eqref{4.10}, we know that
\be\ba\label{4.11}
\lim_{N\rightarrow\infty}(|IV|+|V|)=0.
\ea\ee
Passing to the limit of $N$ in \eqref{4.1}, we conclude the energy equality \eqref{EENS} by
\eqref{4.5}, \eqref{4.7} and \eqref{4.11}.

(2) Case 2: $u_{h}\in L^{p_{1}}(0,T;L^{q_{1}}(\mathbb{R}^{3}))$ for $\f{2}{p_{1}}+\f{3}{q_{1}}=1$, $q_{1}\geq 3.$

 We derive from the interpolation inequality and Sobolev inequality that
\be\ba\label{4.110}
\|u_{h}\|_{L^{\f{2p_{1}}{p_{1}-2}}(0,T;L^{\f{6p_{1}}{p_{1}+4}}(\mathbb{R}^{3}))}
\leq&  \|u_{h}\|^{  \f{ 2}{p_1}}_{L^{\infty}(0,T;L^{2}(\mathbb{R}^{3}))}\|u_{h}\|_{L^{2}(0,T;L^{6}(\mathbb{R}^{3}))}
^{  \f{p_1-2}{p_1}}
\\
\leq& C  \|u_{h}\|^{  \f{ 2}{p_1}}_{L^{\infty}(0,T;L^{2}(\mathbb{R}^{3}))}\|\nabla u_{h}\|_{L^{2}(0,T;L^{2}(\mathbb{R}^{3}))}^{  \f{p_1-2}{p_1}},
\ea\ee
and
\be\ba
\|u_{3}\|_{L^{\f{2p_{1}}{p_{1}-2}}(0,T;L^{\f{6p_{1}}{p_{1}+4}}(\mathbb{R}^{3}))}
\leq&  \|u_{3}\|^{  \f{ 2}{p_1}}_{L^{\infty}(0,T;L^{2}(\mathbb{R}^{3}))}\|u_{3}\|_{L^{2}(0,T;L^{6}(\mathbb{R}^{3}))}
^{  \f{p_1-2}{p_1}}\\
\leq& C  \|u_{3}\|^{  \f{ 2}{p_1}}_{L^{\infty}(0,T;L^{2}(\mathbb{R}^{3}))}\|\nabla u_{3}\|_{L^{2}(0,T;L^{2}(\mathbb{R}^{3}))}^{  \f{p_1-2}{p_1}}.
\ea\ee
This together with
Lemma \ref{lem2.1} yields that
\be\ba\label{4.130}
&\lim_{N\rightarrow\infty}
\|S_{N}(u_{h}u_{h} )-S_{N} u_{h}S_{N}u_{h}\|_{L^{2}(0,T;L^{2}(\mathbb{R}^{3}))}=0,\\
&\lim_{N\rightarrow\infty}
\|S_{N}(u_{3}u_{h} )-S_{N} u_{3}S_{N}u_{h}\|_{L^{2}(0,T;L^{2}(\mathbb{R}^{3}))}=0,
\ea\ee
and
\be\ba\label{4.140}
\lim_{N\rightarrow\infty}\|S_{N}(u_{3}\partial_{1}u_{3} )-S_{N} u_{3}S_{N}\partial_{1}u_{3}\|_{L^{\f{p_{1}}{p_{1}-1}}(0,T;L^{\f{q_{1}}{q_{1}-1}}(\mathbb{R}^{3}))}
=0,
\ea\ee
where
$u_{h}\in L^{p_{1}}(0,T;L^{\f{3p_{1}}{p_{1}-2}}(\mathbb{R}^{3}))$ was used.

It follows from  \eqref{4.110}-\eqref{4.130} that
$$\ba
\lim_{N\rightarrow\infty}(|I |+|II|+|III|)=0.
\ea$$
A combination of \eqref{4.9} and \eqref{4.140} implies that
$$\ba
\lim_{N\rightarrow\infty}(|IV|+|V|)=0.
\ea$$
Since Lemma \ref{lem2.1} is valid for  $ L^{\infty}(0,T;L^{3 }(\mathbb{R}^{3})) $ and $ L^{2}(0,T;L^{6}(\mathbb{R}^{3})) $, the above proof still holds for the limiting case
$u_{h}\in L^{\infty}(0,T;L^{3}(\mathbb{R}^{3}))$.

(3) Case 3: $\nabla u_{h}\in L^{p }(0,T;L^{q }(\mathbb{R}^{3}))$ for $\f{2}{p }+\f{3}{q }=2$, $\f32\leq q <\infty.$

By classical  Sobolev embedding, we know that $\nabla u_{h}\in L^{p }(0,T;L^{q }(\mathbb{R}^{3}))$ for $\f{2}{p }+\f{3}{q }=2$ and $\f32\leq q < 3$ guarantees
$u_{h}\in L^{p }(0,T;L^{\f{3q}{3-q} }(\mathbb{R}^{3}))$. From the previous case, we immediately  get the energy equality \eqref{EENS}. It suffices to focus on the case for $3\leq q < \infty.$
Applying Lemma \ref{lem2.6new}, we derive from $u_{3}\in L^{\infty}(0,T;L^{2}(\mathbb{R}^{3}))  $  and $\nabla u_{3}\in L^{2}(0,T;L^{2}(\mathbb{R}^{3}))  $ that
$u_{3}\in L^{p}(0,T;B^{0}_{q,1}(\mathbb{R}^{3})).$
In addition, repeating the deduction of \eqref{4.110}   yields that Leray-Hopf weak solutions satisfy
$u \in L^{p}(0,T; L^{q}(\mathbb{R}^{3}))   $ with $\f{2}{p}+\f{3}{q}=\f32.$\\
With the help of H\"older inequality, we infer
\be\ba
&\left|\int_{0}^{T}\int_{\mathbb{R}^{3}} [S_{N}(u_{h}u_{h} )-S_{N} u_{h}S_{N}u_{h} ]\partial_{h} S_{N}u_{h} dxds\right| \\
\leq& \|S_{N}(u_{h}u_{h} )-S_{N} u_{h}S_{N}u_{h} \|_{ L^{\f{2q }{3}}(0,T;L^{\f{ q}{q-1}} (\mathbb{R}^{3}))} \|\partial_{h} S_{N}u_{h}\|_{L^{\f{2q}{2q-3}}(0,T;L^{q}  (\mathbb{R}^{3}))}.
\ea\ee

Employing Lemma \ref{lem2.1} for $u_{h}\in L^{\f{4q }{3}}(0,T;L^{\f{ 2q}{q-1}} (\mathbb{R}^{3}))$, we notice that
\be\ba \lim_{N\rightarrow\infty}\|S_{N}(u_{h}u_{h} )-S_{N} u_{h}S_{N}u_{h} \|_{ L^{\f{2q }{3}}(0,T;L^{\f{ q}{q-1}} (\mathbb{R}^{3}))}=0,
\ea\ee
from which it follows that
\be\ba
 \lim_{N\rightarrow\infty}|I|=0.
\ea\ee
By the same taken, we find
\be\ba
 \lim_{N\rightarrow\infty}|III|+|IV|+|V|=0.
\ea\ee
It remains to prove that
\be\ba\label{4.190}
 \lim_{N\rightarrow\infty}|II |=0.
\ea\ee

By means of the H\"older inequality, we observe that
\begin{equation}
\label{4.15}
	\begin{aligned}
	 &\B| \int_{\mathbb{R}^{3}}\big[S_{N}(u_{h}u_{3}) -S_{N}u_{h} S_{N}u_{3} \big]\nabla  S_{N}u_{3} dx \B|\\
\leq	& \|S_{N}(u_{h}u_{3})  -S_{N}u_{h}S_{N}u_{3}\|_{ L^{\f{2q}{q+1}} (\mathbb{R}^{3})}  \|\nabla^{k}_{x}S_{N}u_{3}\|_{ L^{\f{2q}{q-1}}  (\mathbb{R}^{3})}.
\end{aligned}\end{equation}
Following the same path of \eqref{ceti}, we notice that
\be\ba\label{4.16}
&S_{N}(u_{h}u_{3}) -S_{N}u_{h} S_{N}u_{3} \\
=&  2^{ 3N}\int_{\mathbb{R}^n}\tilde{h}(2^{N}y)[u_{h}(x-y)-u_{h}(x)]
[u_{3}(x-y)-u_{3}(x)]dy-(u_{h}-S_{N}u_{h})(u_{3}-S_{N}u_{3}).
\ea\ee
In the light  of  the Minkowski inequality, we infer that
$$\ba\label{4.17}
 &\|S_{N}(u_{h}u_{3})  -S_{N}u_{h}S_{N}u_{3}\|_{    L^{\f{2q}{q+1}}(\mathbb{R}^{3})}
 \\\leq& 2^{ 3N}\int_{\mathbb{R}^n}|\tilde{h}(2^{N}y)|\|u_{h}(x-y)-u_{h}(x) \|_{ L^{q}  (\mathbb{R}^{3})}
 \| u_{3}(x-y)-u_{3}(x) \|_{ L^{\f{2q}{q-1}} (\mathbb{R}^{3}) }dy
\\&+\| u_{h}-S_{N}u_{h} \|_{ L^{q}  (\mathbb{R}^{3})} \| u_{3}-S_{N}u_{3}\|_{L^{\f{2q}{q-1}} (\mathbb{R}^{3})}\\
:=&\,II_{1}+II_{2}.
\ea$$
Thanks to the
mean value theorem, we have
\be\ba\label{4.18}
&\|u_{h}(x-y)-u_{h}(x)\|_{ L^{q} (\mathbb{R}^{3})}\leq C |y| \|  u_{h}\|_{W^{1, q} (\mathbb{R}^{3})}.
\ea\ee
It is  apparent  that
\be\ba\label{4.19}
\| u_{3}(x-y)-u_{3}(x) \|_{ L^{\f{2q}{q-1}} (\mathbb{R}^{3}) }\leq C\| u_{3}  \|_{ L^{\f{2q}{q-1}} (\mathbb{R}^{3}) }.
\ea\ee
A combination of  \eqref{4.18} and
\eqref{4.19} enables us to conclude by
$$ \sup_{N\in\mathbb{Z}} 2^{ 3N}\int_{\mathbb{R}^3}  |\tilde{h}(2^{N}y)|  2^{N}|y| dy<\infty $$
that
\be\label{4.20}
II_{1}\leq C 2^{-   N}\|u_{h}\|_{W^{1, q} (\mathbb{R}^{3})} \| u_{3}  \|_{ L^{ \f{2q}{q-1}} (\mathbb{R}^{3})}.
\ee
In view of $B^{0}_{p,2}\subseteq L^{p}(\mathbb{R}^{n}),$ $ p\geq2$ and Bernstein inequality,
we discover that,  for $N>1$,
\be\ba\label{4.22}
\| u_{h}-{S}_{N}u_{h}\|_{ L^{q}  (\mathbb{R}^{3})}\leq& C    \sum_{j\geq N }\| \Delta_{j} u_{h} \|_{L^{q}(\mathbb{R}^{3})}
\\
\leq
&C 2^{- N} \| u_{h}\|_{W^{1,q} (\mathbb{R}^{3})}.
\ea\ee
It is   evident  that
 \be\ba\label{4.23}
\| u_{3}-{S}_{N} u_{3} \|_{ L^{\f{2q}{q-1}}  (\mathbb{R}^{3})}
\leq
 C \| u_{3}  \|_{ L^{\f{2q}{q-1}}  (\mathbb{R}^{3})}.
\ea\ee
Combining \eqref{4.22} and \eqref{4.23}, we obtain
\be\label{4.24}
II_{2}\leq C 2^{-   N}\|u_{h}\|_{W^{1, q} (\mathbb{R}^{3})} \| u_{3}  \|_{ L^{\f{2q}{q-1}} (\mathbb{R}^{3})}.\ee

Hence, we derive from  the Young inequality and the Bernstein inequality  that
 \be\label{4.25}\ba
\|\nabla S_{N}u_{3}\|_{ L^{\f{2q}{q-1}}  (\mathbb{R}^{3})} \leq& C   \sum_{j< N }2^{j}\|   \Delta_{j} u_{3}  \|_{L^{\f{2q}{q-1}}(\mathbb{R}^{3})}
\\
\leq
& C 2^{N}  \sum_{j< N }2^{j-N}\|   \Delta_{j} u_{3}  \|_{L^{\f{2q}{q-1}}(\mathbb{R}^{3})}
\\
\leq& C 2^{   N} \Gamma^{0}\ast d_{3} ,
\ea\ee
where
$$
\Gamma^{0}(j)=\left\{\ba\label{1.10}
&2^{- j},~~~\text{if}~j>0; \\
&2^{j },~~~~~\text{if}~j\leq 0,
\ea\right.$$
and
$$
d_{3}(j)=\| {\Delta}_{j}u_3\|_{L^{\f{2q}{q-1}} (\mathbb{R}^{3})}.
$$
 Collecting the above estimates, we end up with
 $$\ba
 \B| \int_{\mathbb{R}^{3}}\big[S_{N}(u_{h}u_{3}) -S_{N}u_{h} S_{N}u_{3} \big]\nabla  S_{N}u_{3} dx \B|
 \leq
 C  \|u_{h}\|_{W^{1, q} (\mathbb{R}^{3})} \| u_{3}  \|_{ L^{\f{2q}{q-1}} (\mathbb{R}^{3})}\Gamma^{0}\ast d_{3} (N).
 \ea
 $$
 Performing a time integration, we discover that
 $$\ba
 &\int_{0}^{T}\B| \int_{\mathbb{R}^{3}}\big[S_{N}(u_{h}u_{3}) -S_{N}u_{h} S_{N}u_{3} \big]\nabla  S_{N}u_{3} dx \B|dt\\
 \leq&
 C \int_{0}^{T}\|u_{h}\|_{W^{1, q} (\mathbb{R}^{3})} \| u_{3}  \|_{ L^{\f{2q}{q-1}} (\mathbb{R}^{3})}\Gamma^{0}\ast d_{3} (N) dt\\
 \leq&
 C \B(\int_{0}^{T}\|u_{h}\|^{\f{2q}{2q-3}}_{W^{1, q} (\mathbb{R}^{3})} dt\B)^{\f{2q-3}{2q}}\B(\int_{0}^{T}\| u_{3}  \|^{\f{4q}{3}} _{ L^{\f{2q}{q-1}} (\mathbb{R}^{3})}dt\B)^{\f{3}{4q}}\B(\int_{0}^{T}(\Gamma^{0}\ast d_{3} )^ {\f{4q}{3}}dt\B)^{\f{3}{4q}}.
 \ea
 $$
 Before going further, we shall prove that $  u_{h}\in L^{p }(0,T;W^{1, q }(\mathbb{R}^{3}))$ for $\f{2}{p }+\f{3}{q }=2$ and $3\leq q < \infty$.
It follows from Gagliardo-Nirenberg inequality that, for $q>2$
$$
\|u_{h}\|_{L^{q}(\mathbb{R}^{3})}\leq C \|u_{h}\|^{\f{2q}{5q-6}}_{L^{2}(\mathbb{R}^{3})}\|\nabla u_{h}\|_{L^{q}(\mathbb{R}^{3})}^{\f{3q-6}{5q-6}},
$$
which turns out that
$$
\int_{0}^{T}\|u_{h}\|^{p}_{L^{q}(\mathbb{R}^{3})}ds\leq C \|u_{h}\|^{\f{2q}{5q-6}}_{L^{\infty}(0,T; L^{2}(\mathbb{R}^{3}))}\int_{0}^{T}\|\nabla u_{h}\|_{L^{q}(\mathbb{R}^{3})}^{\f{p(3q-6)}{5q-6}}ds.
$$
It has been shown that $  u_{h}\in L^{\f{2q}{2q-3}}(0,T;W^{1, q }(\mathbb{R}^{3}))$ for   $3\leq q < \infty$. With this in hand, owing to  $u_{3}\in L^{\f{4q}{3}} (0,T;B^{0}_{\f{2q}{q-1},1}
  ) $ and the Lebesgue dominated convergence theorem,  we confirm \eqref{4.190}. The proof of this part is completed.

(4)
We derive from $u_{h}\in L^{2}(0,T;H^{1}(\mathbb{R}^3))$ and Sobolev embedding theorem that $u\in L^{2}(0,T;H^{\gamma}(\mathbb{R}^3))$ for $0\leq\gamma\leq1$. Since $H^{\gamma}(\mathbb{R}^3)\approx B^{\gamma}_{2,2}$, we conclude  by  $B^{\gamma}_{2,2}\subseteq B^{\gamma-\f12}_{3,2}$ that $u_{h}\in L^{2}(0,T; B^{\gamma-\f12}_{3,2})$. To proceed further, we set $\gamma-\f12=\alpha$, which means $\alpha\leq\f12.$ This together with the Lebesgue dominated convergence theorem and $ u_{h}\in L^{3}(0,T;B_{3,\infty}^{\alpha})$
means that $ u_{h}\in L^{3}(0,T;B_{3,c(\mathbb{N})}^{\alpha})$.
Theorem \ref{the1.4} is thus proved.
\end{proof}

\begin{proof}[Proof of Corollary \ref{coro1.5}]
The  Gagliardo-Nirenberg   inequality  guarantees that,  for  $\frac{2}{p}+\frac{2}{q}=1$ with $q\geq 4$
\be\label{gn1}
\begin{aligned}
\|f\|_{L^{4}(0,T;L^{4}(\mathbb{R}^{3}) ))}  \leq& C\|f\|_{L^{\infty}\left(0,T;L^{2}(\mathbb{R}^{3})\right)}^{\frac{(q-4)}{2 q-4}}\|f\|_{L^{p}(0,T;L^q(\mathbb{R}^{3}))}^{\frac{q}{2q-4}}.
\end{aligned}
\ee
Using the Gagliardo-Nirenberg   inequality again, we know that,  for $\frac{1}{p}+\frac{3}{q}=1$ with $3<q<4$
\be\label{gn2}
\begin{aligned}
\|f\|_{L^{4}(0,T;L^{4}(\mathbb{R}^{3}))}
&\leq C\|f\|_{L^2(0,T;L^6(\mathbb{R}^{3}))}^{\frac{3(4-q)}{2(6-q)}}
\|f\|_{L^p(0,T;L^q(\mathbb{R}^{3}))}^{\frac{q}{2(6-q)}}
\\
&\leq C\|\nabla f\|_{L^2(0,T;L^2(\mathbb{R}^{3}))}^{\frac{3(4-q)}{2(6-q)}}
\|f\|_{L^p(0,T;L^q(\mathbb{R}^{3}))}^{\frac{q}{2(6-q)}}.
\end{aligned}
\ee
(1)   Due to hypothesis and inequality \eqref{gn1}, we observe that $ u_{h}\in L^{4}(0,T;L^{4}(\mathbb{R}^{3}))$. We  finish the proof of this part by Theorem \ref{the1.4} immediately.\\
(2) It suffices to apply  \eqref{gn2} and assumption to get $ u_{h}\in L^{4}(0,T;L^{4}(\mathbb{R}^{3}))$. Theorem \ref{the1.4} helps us to prove this case.\\
(3) From \eqref{gn1} and \eqref{gn2}, we have $u_i, u_j\in L^{4}(0,T;L^{4}(\mathbb{R}^{3})).$ The classical Lions' energy class entails the proof of this part. The corollary is proved.
\end{proof}

\section*{Acknowledgements}
 Wang was sponsored by Natural Science Foundation of
Henan Province (No. 232300421077)  and supported by  the National Natural
Science Foundation of China under grant (No. 11971446 and No. 12071113). Wei was partially supported by the National Natural Science Foundation of China under grant (No. 12271433).
 Ye was
partially sponsored by the Training Program for Young Backbone Teachers in
Higher Education Institutions of Henan Province (2024GGJS022).





\end{document}